\documentclass[a4paper,10.5pt]{article}
\usepackage{amsmath,amssymb,amsthm}
\usepackage{cases}
\usepackage{color}
\theoremstyle{definition}
 \newtheorem{dfn}{Definition}[section]
 \newtheorem{remark}[dfn]{Remark}

\theoremstyle{plain}
 \newtheorem{thm}[dfn]{Theorem}
   
 \newtheorem{lem}[dfn]{Lemma}
 \newtheorem{cor}[dfn]{Corollary}

\numberwithin{equation}{section}

\newcommand{\bA}{{\bold A}}

\newcommand{\bD}{{\bold D}}

\newcommand{\bG}{{\bold G}}
\newcommand{\bH}{{\bold H}}

\newcommand{\bI}{{\bold I}}

\newcommand{\bU}{{\bold U}}

\newcommand{\bV}{{\bold V}}
\newcommand{\bW}{{\bold W}}

\newcommand{\DV}{{\rm Div}\,}
\newcommand{\dv}{\, {\rm div}\,}

\newcommand{\BR}{{\Bbb R}}

\newcommand{\BN}{{\Bbb N}}

\newcommand{\BK}{{\Bbb K}}

\newcommand{\CA}{{\mathcal A}}
\newcommand{\CB}{{\mathcal B}}

\newcommand{\CD}{{\mathcal D}}

\newcommand{\CF}{{\mathcal F}}
\newcommand{\CI}{{\mathcal I}}

\newcommand{\CL}{{\mathcal L}}

\newcommand{\CR}{{\mathcal R}}
\newcommand{\CS}{{\mathcal S}}
\newcommand{\CT}{{\mathcal T}}

\newcommand{\CP}{{\mathcal P}}

\newcommand{\CX}{{\mathcal X}}

\newcommand{\bff}{{\bold f}}
\newcommand{\bv}{{\bold v}}
\newcommand{\bu}{{\bold u}}

\newcommand{\bg}{{\bold g}}

\newcommand{\bQ}{{\bold Q}}
\newcommand{\bP}{{\bold P}}
\newcommand{\bS}{{\mathbb S}}
\newcommand{\BS}{{\bold S}}
\newcommand{\bF}{{\bold F}}

\newcommand{\pd}{\partial}

\newcommand{\R}{\mathbb{R}}
\newcommand{\N}{\mathbb{N}}
\newcommand{\Z}{\mathbb{Z}}
\newcommand{\C}{\mathbb{C}}

\newcommand{\vp}{\varphi}

\newcommand{\fp}{{\frak p}}
\newcommand{\tr}{\mathrm{tr}}
\renewcommand{\Re}{{\rm Re}~}

\title{The global well-posedness for the Q-tensor model of nematic liquid crystals in the half-space}
\author{
Daniele Barbera,
\thanks{Department of Mathematical Sciences "Giuseppe Luigi Lagrange", Politecnico di Torino, 
Corso Duca degli Abruzzi 24, 10129 Torino, Italy
\endgraf 
e-mail address: daniele.barbera96@gmail.com
\endgraf
Partially supported by the project E53D23005450006 “Nonlinear dispersive equations in presence of singularities”- funded by European Union– Next Generation EU within the PRIN 2022 program (D.D.
 104- 02/02/2022 Ministero dell'Universit\`a e della Ricerca)
 and
 INDAM, GNAMPA group}
\enskip \enskip
Miho Murata
\thanks{Department of Mathematical and System Engineering,
Faculty of Engineering,
Shizuoka University, 
3-5-1 Johoku, Chuo-ku, Hamamatsu-shi, Shizuoka,
432-8561, Japan
\endgraf
e-mail address: murata.miho@shizuoka.ac.jp
\endgraf
Partially supported by JSPS Grant-in-Aid for Early-Career Scientists 21K13819 and Grant-in-Aid for Scientific Research (B) 23K22405
}
\enskip and \enskip
Yoshihiro Shibata
\thanks{Emeritus Professor of Waseda University, 
Waseda University, 3-4-1 Ohkubo Shinjuku-ku Tokyo, 169-8555, Japan
\endgraf
Adjunct faculty member in the Department of Mechanical Engineering and Materials Science, University of Pittsburgh, United States of America
\endgraf
e-mail address: yshibata325@gmail.com
\endgraf
Partially supported by JSPS Grant-in-Aid for Scientific Research (B) 23K22405}
}
\date{}

\begin{document}
\maketitle
\begin{abstract}
In this paper, we consider the Q-tensor model of nematic liquid crystals, which couples the Navier–Stokes equations with a parabolic-type equation describing the evolution of the directions of the anisotropic molecules, in the half-space.
The aim of this paper is to prove the global well-posedness for the Q-tensor model in the $L_p$-$L_q$ framework.
Our proof is based on the Banach fixed point argument.
To control the higher-order terms of the solutions, we prove the weighted estimates of the solutions for the linearized problem by the maximal $L_p$-$L_q$ regularity.
On the other hand, the estimates for the lower-order terms are obtained by the analytic semigroup theory.
Here, the maximal $L_p$-$L_q$ regularity and the generation of an analytic semigroup are provided by the $\CR$-solvability for the resolvent problem arising from the Q-tensor model.
It seems to be the first result to discuss the unique existence of a global-in-time solution for the Q-tensor model in the half-space. 
\end{abstract}

\section{Introduction}
In the Landau-De Gennes theory of nematic liquid crystals (c.f. \cite{DG, M}), the local orientation and degree of order of liquid crystal molecules are represented by a symmetric and traceless matrix order parameter, which is called the {\it Q-tensor}.
The Beris-Edwards model \cite{BE} is known as one of the models for liquid crystal flows in the context of continuum mechanics. 
The model couples the Navier–Stokes equations with a reaction–diffusion–convection equation for Q-tensor
describing the evolution of the directions of the anisotropic molecules.
From this observation, the Beris-Edwards model is also called the Q-tensor model of liquid crystals.

In this paper, we consider the global well-posedness for the Q-tensor model of liquid crystals in $\R^N_+$, $N \ge 2$.
\begin{equation}\label{nonlinear0}
\left\{
\begin{aligned}
	&\pd_t \bu + (\bu \cdot \nabla) \bu + \nabla \fp
	= \Delta \bu +\DV (\tau(\bQ)+\sigma(\bQ)), 
	\enskip \dv \bu=0
	& \quad&\text{in $\R^N_+$}, \enskip t \in \R_+, \\
	&\pd_t \bQ + (\bu \cdot \nabla) \bQ - \BS(\nabla \bu, \bQ)
	=\bH & \quad&\text{in $\R^N_+$}, \enskip t \in \R_+,\\
	&\bu = 0, \enskip \pd_N \bQ=0& \quad&\text{on $\R^N_0$, $t\in \R_+$},\\
	&(\bu, \bQ)|_{t=0} = (\bu_0, \bQ_0)& \quad&\text{in $\R^N_+$},
\end{aligned}\right.
\end{equation}
where $\bu = \bu(x, t) = (u_1(x, t), \ldots, u_N(x, t))^\mathsf{T} \footnote{$\bA^\mathsf{T}$denotes the transpose of $\bA$.}$ is the fluid velocity,
$\bQ = \bQ(x, t)$ is a symmetric and traceless matrix order parameter (i.e., the Q-tensor) describing the alignment behavior of molecule orientations,
and $\fp=\fp(x, t)$ is the pressure.
For a vector-valued function $\bv$ and a $N\times N$ matrix-valued function $\bA$ with the $(i, j)$ components $A_{ij}$,
we set 
\[
	\dv \bv = \sum_{j=1}^N\pd_j v_j, \quad \DV \bA = \left(\sum_{j=1}^N\pd_jA_{1j}, \sum_{j=1}^N\pd_jA_{2j}, \dots, \sum_{j=1}^N\pd_jA_{Nj}\right)^\mathsf{T},
\]
where $\pd_j=\pd/\pd x_j$.
The tensors $\BS(\nabla \bu, \bQ)$, $\tau(\bQ)$, and $\sigma(\bQ)$ are
\begin{align*}
\BS(\nabla \bu, \bQ)&=
(\xi \bD(\bu)+ \bW(\bu)) \left( \bQ + \frac{1}{N} \bI \right)
+\left( \bQ + \frac{1}{N} \bI \right)(\xi \bD(\bu)- \bW(\bu))
-2\xi\left( \bQ + \frac{1}{N} \bI \right) \tr(\bQ \nabla \bu),\\
\tau(\bQ) 
&=2\xi \tr(\bH \bQ)\left( \bQ + \frac{1}{N} \bI \right)
-\xi\left[\bH\left( \bQ + \frac{1}{N} \bI \right) + \left( \bQ + \frac{1}{N} \bI \right) \bH\right]
- L\nabla \bQ \odot \nabla \bQ,
\\
\sigma(\bQ)
&=\bQ \bH - \bH \bQ, 
\end{align*}
where 
\begin{align*}
	&\bD(\bu) = \frac12(\nabla \bu +(\nabla \bu)^\mathsf{T}), \quad
	\bW(\bu) = \frac12(\nabla \bu -(\nabla \bu)^\mathsf{T}),\\
	&(\nabla \bQ \odot \nabla \bQ)_{i j} = \sum^N_{k, \ell = 1} \pd_i Q_{k \ell} \pd_j Q_{k \ell},
\end{align*}
and $\bI$ is the $N \times N$ identity matrix.
A scalar parameter $\xi\in \R$
denotes the ratio between the tumbling 
and the aligning effects that a shear flow would exert over the directors.
Set
\[
	\bH = L\Delta \bQ - a\bQ + b\left(\bQ^2 - (\tr (\bQ^2))\bI/N \right) - c\tr(\bQ^2) \bQ. 
\] 
Note that $\bH$ is derived from the first order 
variation of the Landau-De Gennes free energy functional:
\[
	\CF(\bQ) = \int_{\R^N_+} \left(\frac{L}{2}|\nabla \bQ|^2 + F(\bQ)\right)\,dx,
\]
where 
$L>0$ is the elastic constant.
Hereafter, we set $L=1$ for simplicity. 
Furthermore, $F(\bQ)$ denotes the bulk energy of Landau-de Gennes type:
\[
	F(\bQ)=\frac{a}{2} \tr(\bQ^2) - \frac{b}{3} \tr(\bQ^3) + \frac{c}{4} (\tr(\bQ^2))^2
\]
with a material-dependent and temperature-dependent non-zero constant $a$ and material-dependent positive constants $b$ and $c$.
In addition, we assume that $\xi \neq 0$ and $a>0$ from a mathematical point of view.

The existence of solutions for the Q-tensor model has been discussed in the whole space or in bounded domains.
The existence of weak solutions was studied for $\xi=0$ or $\xi$ sufficiently small in $\R^N$, $N=2, 3$ (e.g., \cite{D, HD, PZ1, PZ2}).
Here, $\xi=0$ means that the molecules only tumble in a shear flow; however, they are not aligned by such a flow.
Abels, Dolzmann, and Liu \cite{Ab1} proved the existence of a strong local solution and 
global weak solutions with higher regularity 
in time, in the case of inhomogeneous mixed Dirichlet/Neumann boundary conditions in a bounded domain
without any smallness assumption on the parameter $\xi$. 
Liu and Wang \cite{LW} improved the spatial regularity of solutions obtained in \cite{Ab1}
and generalized their result to the case of anisotropic elastic energy.
Abels, Dolzmann, and Liu \cite{Ab2} also proved the local well-posedness in a bounded domain with the homogeneous Dirichlet boundary condition for the case $\xi=0$. 
These results \cite{Ab1, Ab2, LW} are obtained in the $L_2$-framework.
In the maximal $L_p$-$L_q$ regularity class,
Xiao \cite{X} proved the global well-posedness in a bounded domain for the case $\xi=0$.
Thanks to the assumption $\xi=0$, the maximal $L_p$-$L_q$ regularity for the Q-tensor model follows from it for the Stokes and parabolic operators.
Hieber, Hussein, and Wrona \cite{HHW} established the global well-posedness in a bounded domain for any $\xi$.
They proved that the linear operator is $\CR$-sectorial by proving that the linear operator is invertible and its numerical range lies in a certain sector, which is based on a classical result for unbounded operators in Hilbert spaces (cf. \cite{K}), which implies that the linear operator has the maximal $L_p$-$L_2$ regularity for $p>4/{4-N}$ with $N=2, 3$.
The whole-space problem was studied by Schonbek and the third author \cite{SS} and the second and third authors \cite{MS}.
For $1< p, q<\infty$, the maximal $L_p$-$L_q$ regularity was proved by the $\CR$-boundedness for the solution operators to the resolvent problem, 
where the resolvent parameter $\lambda$ is far away from the origin.
Furthermore, \cite{SS, MS} proved the decay estimates for the linearized problem based on the decay estimates for the heart semigroup, then the global well-posedness was established in $\R^N$, $N \ge 3$. 

On the other hand,
the half-space problem was first studied by the first and second authors \cite{BM}.
The local well-posedness in the maximal $L_p$-$L_q$ regularity class for the small initial data was obtained by \cite{BM}; however, the global well-posedness is an open problem even in the $L_2$-setting. 

In this paper, we prove 
the global well-posedness for \eqref{nonlinear0} based on the Banach fixed point argument.
Hereafter, we mainly consider the following problem, divided \eqref{nonlinear0} into the linear and the nonlinear terms.
\begin{equation}\label{nonlinear}
\left\{
\begin{aligned}
	&\pd_t\bu -\Delta \bu + \nabla \fp + \beta \DV (\Delta \bQ -a \bQ)=\bff(\bu, \bQ),
	\enskip \dv \bu=0& \quad&\text{in $\R^N_+$}, \enskip t \in \R_+,\\
	&\pd_t \bQ - \beta \bD(\bu) - \Delta \bQ + a \bQ =\bG(\bu, \bQ)& \quad&\text{in $\R^N_+$}, \enskip t \in \R_+,\\
	&\bu = 0, \enskip \pd_N \bQ=0& \quad&\text{on $\R^N_0$, $t\in \R_+$},\\
	&(\bu, \bQ)|_{t=0}=(\bu_0, \bQ_0)& \quad&\text{in $\R^N_+$},
\end{aligned}
\right.
\end{equation}
where 
\begin{align*}
	\beta
	&= 2\xi/N,\\
	\bff(\bu, \bQ)
	&= -(\bu \cdot \nabla) \bu + 
	\DV[2\xi \bold{H}: \bQ(\bQ+\bI/N) - (\xi + 1) \bold{H}\bQ +
	 (1-\xi) \bQ\bold{H} -\nabla \bQ \odot \nabla \bQ]- \beta \DV F'(\bQ),\\
	\bG(\bu, \bQ)
	&= -(\bu \cdot \nabla)\bQ + \xi(\bD(\bu) \bQ + \bQ \bD(\bu))
	+\bW(\bu) \bQ - \bQ \bW(\bu) -2\xi (\bQ +\bI/N) \bQ : \nabla \bu + F'(\bQ)
\end{align*}
Here, $F'(\bQ)$ is the nonlinear term of $\bH$; namely, $F'(\bQ) = b\left(\bQ^2 - (\tr (\bQ^2))\bI/N \right) - c \tr(\bQ^2) \bQ$.
Now, we state our method in more detail.
Let $\bU=(\bu, \bQ)$.
First, we consider the linearized system
\begin{equation}\label{model}
\left\{
\begin{aligned}
	&\pd_t \bU + \CA_q \bU=\bF& \quad&\text{in $\R^N_+$, $t\in \R_+$}, \\
	&\CB \bU = 0& \quad&\text{on $\R^N_0$, $t\in \R_+$},\\
	&\bU(0)=\bU_0 & \quad&\text{in $\R^N_+$},
\end{aligned}\right.
\end{equation}
where $\CA_q$ is a linear operator with a domain $\CD(\CA_q)$ defined in subsection \ref{subsec:main} below, $\CB \bU=(\bu, \pd_N \bQ)$, $\bF=(\bff, \bG)$ and $\bU_0=(\bu_0, \bQ_0)$ are given functions.
Assume that $\CA_q$ has the maximal $L_p$-$L_q$ regularity
and generates an analytic semigroup on the Banach space $\CX_q(\R^N_+)$.
Note that these facts can be proved by the fact that the family of solution operators for the resolvent problem arising from \eqref{model} is the $\CR$-bounded when the resolvent parameter is close to the origin (cf. \cite{BM2025}).
Then \eqref{model} has a solution $\bU$ satisfying
\begin{equation}\label{mr for model}
	\|(\pd_t, \CA_q) \bU\|_{L_p(\R_+, \CX_q(\R^N_+))} 
	\le C(\|\bU_0\|_{(\CX_q(\R^N_+), \CD(\CA_q))_{1-1/p, p}} + \|\bF\|_{L_p(\R_+, \CX_q(\R^N_+))}).
\end{equation}
Let us consider the weighted estimates of the higher-order terms for \eqref{model}.
Multiply $t$ with \eqref{model}, $\bU$ satisfies
\[
\left\{
\begin{aligned}
	&\pd_t (t\bU) + \CA_q (t\bU)=t\bF + \bU& \quad&\text{in $\R^N_+$, $t\in \R_+$}, \\
	&\CB (t\bU) = 0& \quad&\text{on $\R^N_0$, $t\in \R_+$},\\
	&t\bU(0)=0 & \quad&\text{in $\R^N_+$},
\end{aligned}\right.
\]
then it holds by \eqref{mr for model} that
\[
	\|(\pd_t, \CA_q) t\bU\|_{L_p(\R_+, \CX_q(\R^N_+))} \le C(\|t\bF\|_{L_p(\R_+, \CX_q(\R^N_+))} + \|\bU\|_{L_p(\R_+, \CX_q(\R^N_+))}).
\]
The estimates of the lower-order term $\|\bU\|_{L_p(\R_+, \CX_q(\R^N_+))}$ are provided by the boundedness and the decay estimate of the semigroup,
which is obtained by the resolvent estimates.
Then we arrive at the weighted estimates of the higher-order terms
\[
	\|(1+t)(\pd_t, \CA_q) \bU\|_{L_p(\R_+, \CX_q(\R^N_+))} \le C\left(\CI + \sum_{r \in \{q, \widetilde q\}}\|(1+t)\bF\|_{L_p(\R_+, \CX_r(\R^N_+))}\right)
\]
for some $p$, $q$, and $\widetilde q$,
where $\CI = \|\bU_0\|_{(\CX_q(\R^N_+), \CD(\CA_q))_{1-1/p, p}} + \|\bU_0\|_{\CX_{\widetilde q}}$.
Note that the additional regularity for the initial data is not necessary to obtain the weighted estimates. 
This approach for the linear system differs from \cite{SS, MS}.
Next, we consider \eqref{nonlinear}.
Set $\bF(\bU) = (\bff(\bU), \bG(\bU))$ and 
\[
	E(\bU) = \|(1+t)(\pd_t, \CA_q) \bU\|_{L_p(\R_+, \CX_q(\R^N_+))} + \|\bU\|_{L_p(\R_+, \CX_q(\R^N_+))} + \|\bU\|_{L_\infty(\R_+, \CX_q(\R^N_+))}.
\]
Since nonlinear terms have the quasi-linear term and the lower-order terms,
$\|(1+t)\bF(\bU)\|_{L_p(\R_+, \CX_q(\R^N_+))}$ is controlled by $E(\bU)$; therefore, we can apply the Banach fixed point argument for small initial data.
This method may be applied to other parabolic equations if the linear operator has the maximal $L_p$-$L_q$ regularity and generates an analytic semigroup.

This paper is organized as follows: Section \ref{sec:main} states the global well-posedness in the maximal $L_p$-$L_q$ regularity class as the main theorem in this paper.
In addition, we state the existence of the $\CR$-bounded solution operator families for the resolvent problem, which is the basis of the linear theory in our method. 
Section \ref{sec:mr} proves the maximal $L_p$-$L_q$ regularity estimates for the linearized problem.
The proof is divided into two parts: the estimates for the homogeneous system and the linear equations with zero initial conditions.
The first part is obtained by the $\CR$-solvability for the resolvent problem and  the Weis operator-valued Fourier multiplier theorem, while
the second part is proved by semigroup theory and the real interpolation argument. 
Section \ref{sec:weight} proves the weighted estimates of the higher-order terms for the linearized problem.
The estimates of the lower-order terms for the linearized problem can be obtained from the semigroup theory.
 Finally, Section \ref{sec:global} proves the global well-posedness for the small initial data based on the Banach fixed point argument.

\section{Main Theorem}\label{sec:main}
In this section, we state the global well-posedness for \eqref{nonlinear0} in the maximal $L_p$-$L_q$ regularity class.
\subsection{Notation}
Let us summarize several symbols and functional spaces used 
throughout the paper.
$\N$, $\R$, $\C$, and $\Z$ denote the sets of 
all natural numbers, real numbers, complex numbers, and integer number, respectively. 
We set $\N_0=\N \cup \{0\}$ and $\R_+ = (0, \infty)$. 
Let $q'$ be the dual exponent of $q$
defined by $q' = q/(q-1)$
for $1 < q < \infty$. 
For any multi-index $\alpha = (\alpha_1, \ldots, \alpha_N) 
\in \N_0^N$, we write $|\alpha|=\alpha_1+\cdots+\alpha_N$ 
and $D_x^\alpha=\pd_1^{\alpha_1} \cdots \pd_N^{\alpha_N}$ 
with $x = (x_1, \ldots, x_N)$ and $\pd_j=\pd/\pd x_j$. 
For $k \in \N_0$, scalar function $f$, $N$ vector-valued function $\bg$, 
and $N \times N$ matrix-valued function $\bG$, we set
\begin{gather*}
\nabla^k f = (D_x^\alpha f \mid |\alpha|=k),
\enskip \nabla^k \bg = (D_x^\alpha g_j \mid |\alpha|=k, \enskip j = 1,\ldots, N),\\
\nabla^k \bG = (D_x^\alpha G_{i j} \mid |\alpha|=k, \enskip i, j = 1,\ldots, N).
\end{gather*} 

Hereafter, we use small boldface letters, e.g. $\bff$ to 
denote vector-valued functions and capital boldface letters, e.g. $\bG$
to denote matrix-valued functions, respectively. 
The letter $C$ denotes generic constants, and the constant $C_{a,b,\ldots}$ depends on $a,b,\ldots$. 
The values of constants $C$ and $C_{a,b,\ldots}$ 
may change from line to line. 

For $N \in \N$, the Fourier transform $\CF$ and 
its inverse transform $\CF^{-1}$ are defined by 
\[
	\CF[f](\xi) = \int_{\R^N}e^{-ix\cdot\xi}f(x)\,dx, \quad
	\CF^{-1}_\xi[g](x) = \frac{1}{(2\pi)^N}\int_{\R^N}
	e^{ix\cdot\xi}g(\xi)\,d\xi. 
\]
Furthermore, the Laplace transform $\CL$ and its inverse transform $\CL^{-1}$
are defined by
\[
	\CL[f](\lambda) = \int_{\R}e^{-\lambda t}f(t)\,dt, \quad
	 \CL^{-1}[g](t)= \frac{1}{2\pi}\int_{\R}e^{\lambda t}g(\tau)\,d\tau,
\]
where $\lambda = \gamma + i\tau \in\C$, which are written by Fourier transform
and its inverse transform in $\R$ as 
\[
	\CL[f](\lambda) = \CF[e^{-\gamma t}f(t)](\tau),
	\quad \CL^{-1}[g](t) = e^{\gamma t}\CF^{-1}[g](\tau).
\]

For $N \in \N$, $1 \le p \le \infty$, and $m \in \N$,
$L_p(\R^N_+)$ and $H_p^m(\R^N_+)$ 
denote the Lebesgue space and the Sobolev space in $\R^N_+$;
while $\|\cdot\|_{L_q(\R^N_+)}$ and $\|\cdot\|_{H_q^m(\R^N_+)}$
denote their norms, respectively. 
In addition, $B^s_{q, p}(\R^N_+)$ is the Besov space in $\R^N_+$ for $1 < q < \infty$ and $s \in \R$ with the norm $\|\cdot\|_{B_{q, p}^s(\R^N_+)}$.
The $d$-product space of $X$ is defined by 
$X^d=\{f=(f, \ldots, f_d) \mid f_i \in X \, (i=1,\ldots,d)\}$,
while its norm is denoted by 
$\|\cdot\|_X$ instead of $\|\cdot\|_{X^d}$ for the sake of 
simplicity. 
The usual Lebesgue space and the Sobolev space of $X$-valued functions 
defined on time interval $I$ are denoted by $L_p(I, X)$ and $H^m_p(I, X)$ with $1 \le p \le \infty$ and $m \in \N$; 
while $\|\cdot\|_{L_p(I, X)}$, $\|\cdot\|_{H_p^m(I, X)}$
denote their norms, respectively.

For Banach spaces $X$ and $Y$, $\CL(X,Y)$ denotes the set of 
all bounded linear operators from $X$ into $Y$,
$\CL(X)$ is the abbreviation of $\CL(X, X)$, and 
$\rm{Hol}\,(U, \CL(X,Y))$ 
 the set of all $\CL(X,Y)$ valued holomorphic 
functions defined on a domain $U$ in $\C$. 
For the interpolation couple $(X, Y)$ of Banach spaces, $0<\theta<1$, and $1 \le p \le \infty$,
the real interpolation space is denoted by $(X, Y)_{\theta, p}$.

For Banach spaces $X$ and $N \in \N$, let $\CS(\R^N, X)$ be the Schwartz class
of $X$-valued functions on $\R^N$,
while $\CS'(\R^N, X)$ be the space of $X$-valued tempered distributions; namely, $\CS'(\R^N, X)= \CL(\CS(\R^N, X), X)$. 
For simplicity, we write $\CS(\R^N) = \CS(\R^N, \BK)$
$\CS'(\R^N) = \CS'(\R^N, \BK)$, where $\BK=\R$ or $\C$.

\subsection{The homogeneous Sobolev and Besov spaces}
In this subsection, we introduce the homogeneous Sobolev and Besov spaces in $\R^N_+$.
For $1 < q < \infty$, and $s \in \N$, the homogeneous Sobolev space $\dot H^s_q(\R^N)$ is defined as
\[
	\dot H^s_q(\R^N) = \{f \in \CS'(\R^N) \setminus \CP(\R^N) \mid \|f\|_{\dot H^s_q(\R^N)} < \infty\},
\]
where we have set
\[
	\|f\|_{\dot H^s_q(\R^N)} = \|\CF^{-1} [|\xi|^s \CF[f](\xi)]\|_{L_q(\R^N)}.
\]
Here, $\CP(\R^N)$ denotes the set of all polynomials.

Let us define the homogeneous Besov space.
Let $\phi \in \CS(\R^N)$ with ${\rm supp}~ \phi = \{\xi \in \R^N \mid 1/2 \le |\xi| \le 2\}$ such that 
$\sum_{j \in \Z} \phi(2^{-j}\xi) =1$ for any $\phi \in \R^N \setminus \{0\}$. Set $\phi_0 (\xi) = 1-\sum^\infty_{j=1} \phi(2^{-j}\xi)$.
Let $\{\dot \Delta_j\}_{j \in \Z}$ be the homogeneous family of Littlewood-Paley dyadic decomposition operators defined by
\[
	\dot \Delta_j f = \CF^{-1}[\phi(2^{-j}\xi) \CF[f](\xi)]
\]
for $j \in \Z$.
For $1 \le p, q \le \infty$ and $s \in \N$, we set
\[
	\|f\|_{\dot B^s_{q, p}(\R^N)} = \|2^{js} \|\dot \Delta_j f\|_{L_q(\R^N)}\|_{\ell^p(\Z)}.
\]
Then
the homogeneous Besov space $\dot B^s_{q, p}(\R^N)$ is defined as
\[
	\dot B^s_{q, p}(\R^N) = \{f \in \CS'(\R^N) \setminus \CP(\R^N) \mid \|f\|_{\dot B^s_{q, p}(\R^N)} < \infty\},
\]
where $\ell^p$ denotes sequence spaces.

Now, we define the homogeneous Sobolev spaces and the homogeneous Besov spaces in $\R^N_+$.
Let $1 \le p \le \infty$, $1 < q < \infty$, and $s \in \N$. 
For $X \in \{\dot H^s_q, \dot B^s_{q, p} \}$, we define
\[
	X(\R^N_+) = \{g|_{\R^N_+} = f \mid g \in X(\R^N) \}
\]
with the quotient norm $\|f\|_{X(\R^N_+)} = \displaystyle \inf_{\substack{g \in X(\R^N) \\ g|_{\R^N_+} = f}} \|g\|_{X(\R^N)}$. 
In particular, by this definition and $\dot H^2_q(\R^N)^N \cap L_q(\R^N)^N = H^2_q(\R^N)^N$ (cf. \cite[Theorem 6.3.2]{BL}), it holds that
\begin{equation}\label{homo Sobolev}
	\dot H^2_q(\R^N_+)^N \cap L_q(\R^N_+)^N = H^2_q(\R^N_+)^N.
\end{equation}

For simplicity, we set $\dot H^{0, 1}_q(\R^N_+)=L_q(\R^N_+) \times \dot H^1_q(\R^N_+)$.

\subsection{Main Theorem}\label{subsec:main}
To state the main theorem, we introduce some spaces.
Let $\bS_0 \subset \R^{N^2}$ denotes the set of the Q-tensor; namely,
\[
 	\bS_0=\{\bQ \in \R^{N^2} \mid \tr \bQ=0, \enskip \bQ=\bQ^\mathsf{T}\}.
\]

The space for the pressure term and a solenoidal space are defined as 
\begin{align*}
	\widehat H^1_{q, 0}(\R^N_+) &= \{f \in L_{q, {\rm loc}} (\R^N_+) \mid \nabla f \in L_q(\R^N_+), \enskip f = 0 \text{~on~} \R^N_0\},\\
	J_q(\R^N_+) &=\{\bu \in L_q(\R^N_+) \mid (\bu, \nabla \vp) = 0 \quad \forall \vp \in \widehat H^1_{q', 0}(\R^N_+)\}.
\end{align*}

Let us introduce the functional space for the initial data.
Define an operator $\CA_q$ and its domain $\CD(\CA_q)$ as
\begin{align*}
	\CD(\CA_q)
	&=\{(\bu, \bQ)\in (\dot H^2_q(\R^N_+)^N \cap J_q(\R^N_+)) \times (\dot H^3_q(\R^N_+; \bS_0) \cap \dot H^1_q(\R^N_+; \bS_0)) \mid \bu|_{x_N=0} = 0, \enskip \pd_N \bQ|_{x_N=0}=0\},\\
	\CA_q(\bu, \bQ)&=(\Delta \bu-\nabla K(\bu, \bQ)-\beta \DV(\Delta \bQ-a\bQ), \beta \bD(\bu) + \Delta \bQ - a\bQ) \enskip \text{for} \enskip (\bu, \bQ) \in \CD(\CA_q),
\end{align*}
where $p=K(\bu, \bQ)$ is a solution of the weak Dirichlet Neumann problem:
\[
	(\nabla p, \nabla \vp)=(\Delta \bu-\beta \DV(\Delta \bQ-a\bQ), \nabla \vp)
\]
for any $\vp \in \widehat H^1_{q', 0}(\R^N_+)$.
In addition, we set
\[
	\CX_q(\R^N_+)=J_q(\R^N_+) \times \dot H^1_q(\R^N_+; \bS_0).
\]
Then we define
\[
	\CD_{q, p}(\R^N_+)=(\CX_q(\R^N_+), \CD(\CA_q))_{1-1/p, p}.
\]
Taking into account \eqref{homo Sobolev} and 
\begin{equation}\label{interpolation Q}
\begin{aligned}
	&(\dot H^1_q(\R^N_+; \bS_0), \dot H^3_q(\R^N_+; \bS_0) \cap \dot H^1_q(\R^N_+; \bS_0))_{1-1/p, p} \\
	&= (\dot H^1_q(\R^N_+; \bS_0), \dot H^3_q(\R^N_+; \bS_0))_{1-1/p, p} \cap \dot H^1_q(\R^N_+; \bS_0)\\
	&= \dot B^{3-2/p}_{q, p}(\R^N_+; \bS_0) \cap \dot H^1_q(\R^N_+; \bS_0)
\end{aligned}
\end{equation}
(cf. \cite[Proposition B.2.7]{H} and \cite[Proposition 2.10]{SW}),
we have
\[
	\CD_{q, p}(\R^N_+) \subset B^{2(1-1/q)}_{q, p}(\R^N_+)^N \times (\dot B^{3-2/p}_{q, p}(\R^N_+; \bS_0) \cap \dot H^1_q(\R^N_+; \bS_0)).
\]

Let us state the main theorem in this paper.
\begin{thm}\label{thm:global}
Let $N \ge 2$, and let $0<\theta<1/2$.  
Assume that 
\begin{equation}\label{condi}
	\frac{1}{q_0}=\frac{1+2\theta}{N}, \quad 
	\frac{1}{q_1}=\frac{1+\theta}{N}, \quad \frac{1}{q_2}=\frac{\theta}{N}, \quad \frac{1}{p}<\frac{\theta}{2}.
\end{equation}
Let 
\[
	(\bu_0, \bQ_0) \in
	 \bigcap_{i=1}^2 \CD_{q_i, p}(\R^N_+)
	\bigcap (J_{q_0}(\R^N_+)\times \dot H^1_{q_0}(\R^N_+; \bS_0)),
\]
and let 
\begin{align*}
	E(\bu, \bQ)&=\sum^2_{i=1}(\|(1+t)(\pd_t, \nabla^2)(\bu, \bQ)\|_{L_p(\R_+, L_{q_i}(\R^N_+) \times \dot H^1_{q_i}(\R^N_+))}+\|(1+t)\nabla \bQ\|_{L_p(\R_+, L_{q_i}(\R^N_+)}\\
	&\quad +\|(\bu, \bQ)\|_{L_p(\R_+, L_{q_i}(\R^N_+)\times \dot H^1_{q_i}(\R^N_+))}
	+\|(\bu, \bQ)\|_{L_\infty(\R_+, L_{q_i}(\R^N_+)\times \dot H^1_{q_i}(\R^N_+))}).
\end{align*}
Then there exists a small number $\sigma>0$ such that 
\begin{equation}\label{smallness condi1}
	\sum^2_{i=1} \|(\bu_0, \bQ_0)\|_{\CD_{q_i, p}(\R^N_+)} + \|(\bu_0, \bQ_0)\|_{L_{q_0}(\R^N_+) \times \dot H^1_{q_0}(\R^N_+)} \le \sigma^2,
\end{equation}
problem \eqref{nonlinear0} has a unique solution $(\bu, \bQ, \fp)$ with
\begin{align}
	\pd_t \bu &\in \bigcap^2_{i=1} L_p(\R_+, L_{q_i}(\R^N_+)),&\bu &\in \bigcap^2_{i=1} L_p(\R_+, \dot H^2_{q_i}(\R^N_+)),\label{regularity:u}\\
	\pd_t \bQ &\in \bigcap^2_{i=1} L_p (\R_+, \dot H^1_{q_i}(\R^N_+; \bS_0)), & 
	\bQ &\in \bigcap^2_{i=1} L_p (\R_+, \dot H^1_{q_i}(\R^N_+; \bS_0) \cap \dot H^3_{q_i}(\R^N_+; \bS_0)), \nonumber\\
	\nabla \fp & \in \bigcap^2_{i=1} L_p(\R_+, L_{q_i}(\R^N_+)) & & \nonumber
\end{align}
satisfying 
\begin{equation}\label{est:large}
	E(\bu, \bQ) \le \sigma.
\end{equation}
In addition,
there exists a constant $C$ such that
\[
	\|(1+t) \nabla \fp\|_{L_p(\R_+, L_{q_i}(\R^N_+))} \le C \sigma
\]
for $i=1, 2$.
\end{thm}

\begin{remark}
\begin{enumerate}

\item
By \eqref{regularity:u} and \eqref{est:large}, we observe that
\[
	\pd_t \bu \in \bigcap^2_{i=1} L_p(\R_+, L_{q_i}(\R^N_+)), \quad \bu \in \bigcap^2_{i=1} L_p(\R_+, \dot H^2_{q_i}(\R^N_+)) \cap L_p(\R_+, L_{q_i}(\R^N_+)),
\]
together with \eqref{homo Sobolev},
we have 
\[
	\bu \in \bigcap^2_{i=1} H^1_p(\R_+, L_{q_i}(\R^N_+)) \cap L_p(\R_+, H^2_{q_i}(\R^N_+)).
\]

\item
Thanks to \eqref{est:large} and Lemma \ref{sup} below, $\bQ$ satisfies
\[
	\bQ \in L_\infty(\R_+, L_\infty(\R^N_+)).
\]
\end{enumerate}
\end{remark}

\subsection{Preliminary}
First, we recall the definition of the $\CR$-boundedness.
\begin{dfn}\label{dfn2}
A family of operators $\CT \subset \CL(X,Y)$ is called $\CR$-bounded 
on $\CL(X,Y)$, if there exist constants $C > 0$ and $p \in [1,\infty)$ 
such that for any $n \in \BN$, $\{T_{j}\}_{j=1}^{n} \subset \CT$,
$\{f_{j}\}_{j=1}^{n} \subset X$ and sequences $\{r_{j}\}_{j=1}^{n}$
 of independent, symmetric, $\{-1,1\}$-valued random variables on $[0,1]$, 
we have  the inequality:
$$
\bigg \{ \int_{0}^{1} \|\sum_{j=1}^{n} r_{j}(u)T_{j}f_{j}\|_{Y}^{p}\,du
 \bigg \}^{1/p} \leq C\bigg\{\int^1_0
\|\sum_{j=1}^n r_j(u)f_j\|_X^p\,du\biggr\}^{1/p}.
$$ 
The smallest such $C$ is called $\CR$-bound of $\CT$, 
which is denoted by $\CR_{\CL(X,Y)}(\CT)$.
\end{dfn}
\begin{remark}\label{rem:def of rbdd}
The $\CR$-boundedness implies that the uniform boundedness of the operator family $\CT$.
In fact, choosing $m=1$ in Definition \ref{dfn2}, we observed that there exists a constant $C$ such that $\|T f\|_Y \le C \|f\|_X$ holds for any $T \in \CT$ and $f \in X$.
\end{remark}
Second, we state the results for $\CR$-bounded solution operator families for the resolvent problem: 
\begin{equation}\label{r0}
	\left\{
	\begin{aligned}
	&\lambda\bu -\Delta \bu + \nabla \fp + \beta \DV (\Delta \bQ -a \bQ)=\bff,
	\enskip \dv \bu=0& \quad&\text{in $\R^N_+$},\\
	&\lambda \bQ - \beta \bD(\bu) - \Delta \bQ + a \bQ =\bG& \quad&\text{in $\R^N_+$},\\
	&\bu=0, \enskip \pd_N \bQ=0 & \quad&\text{on $\R^N_0$},
	\end{aligned}
	\right.
\end{equation}
where $a > 0$, $\beta \neq 0$,
and $\lambda$ is the resolvent parameter varying in a sector
\[
	\Sigma_{\epsilon} 
	=\{\lambda \in \C\setminus \{0\} \mid |\arg \lambda| < \pi - \epsilon\} 
\]
for $\epsilon_0 < \epsilon < \pi/2$ with $\tan \epsilon_0 \ge |\beta|/\sqrt 2$. 
The following  theorem follows from \cite[Theorem 3.4.5]{BM}, \cite[Theorem 3.3, Remark 3.4, and Theorem 6.1]{BM2025}.

\begin{thm}\label{thm:Rbdd H}
Let $1 < q < \infty$, and let $\epsilon \in (\epsilon_0, \pi/2)$ with $\tan \epsilon_0 \ge |\beta|/\sqrt 2$. 
Let 
\begin{align*}
X_q(\R^N_+)&=L_q(\R^N_+)^N \times L_q(\R^N_+; \R^{N^3}), 
\end{align*}
and let 
$\bF=(\bff, \nabla \bG) \in X_q(\R^N_+)$.
There exist operator families 
\begin{align*}
	&\CA (\lambda) \in 
	{\rm Hol} (\Sigma_{\epsilon}, 
	\CL(X_q(\R^N_+), H^2_q(\R^N_+)^N))\\
	&\CB (\lambda) \in 
	{\rm Hol} (\Sigma_{\epsilon}, 
	\CL(X_q(\R^N_+), H^3_q(\R^N_+; \bS_0)))
\end{align*}
such that 
for any $\lambda = \gamma + i\tau \in \Sigma_{\epsilon}$,
$\bu = \CA (\lambda) \bF$ and 
$\bQ = \CB (\lambda) \bF$
are unique solutions of \eqref{r0},
and 
\begin{align*}
	&\CR_{\CL(X_q(\R^N_+), A_q(\R^N_+))}
	(\{(\tau \pd_\tau)^n \CS_\lambda \CA (\lambda) \mid 
	\lambda \in \Sigma_\epsilon\}) 
	\leq r,\\
	&\CR_{\CL(X_q(\R^N_+), B_q(\R^N_+))}
	(\{(\tau \pd_\tau)^n \CS_\lambda \CB (\lambda) \mid 
	\lambda \in \Sigma_\epsilon\}) 
	\leq r 
\end{align*}
for $\ell = 0, 1,$
where 
$\CS_\lambda = (\nabla^2, \lambda^{1/2}\nabla, \lambda)$,
$A_q(\R^N_+) = L_q(\R^N_+)^{N^3 + N^2+N}$,
$B_q(\R^N_+) = \dot H^1_q(\R^N_+; \R^{N^4}) \times \dot H^1_q(\R^N_+; \R^{N^3}) \times \dot H^1_q(\R^N_+; \bS_0)$,
and $r=r_{N, q}$ is a constant independent of $\lambda$.
\end{thm}
Note that the unique existence of the pressure $\fp$ follows from the unique solvability of the weak Dirichlet Neumann problem (cf. \cite[subsection 5.5]{BM2025}).
Theorem \ref{thm:Rbdd H}, together with Remark \ref{rem:def of rbdd}, implies that the resolvent estimates for \eqref{r0}.
\begin{cor}\label{cor:resolvent}
Let $1 < q < \infty$ and
$\epsilon \in (\epsilon_0, \pi/2)$ with $\tan \epsilon_0 \ge |\beta|/\sqrt 2$.  
Then for any $\lambda \in \Sigma_\epsilon$, $\bff \in L_q(\R^N_+)^N$ and $\bG \in \dot H^1_q(\R^N_+; \bS_0)$, 
there is a unique solution $(\bu, \bQ, \fp)$ for \eqref{r0}, unique up to additive constant on $\fp$, 
with $\bu \in H^2_q(\R^N_+)^N$, $\bQ \in H^3_q(\R^N_+; \bS_0)$, 
$\fp \in \widehat H^1_{q, 0}(\R^N_+)$, and
\begin{equation}\label{rem:resolvent}
\begin{aligned}
&\|(|\lambda|, |\lambda|^{1/2} \nabla, \nabla^2)(\bu, \bQ)\|_{L_q(\R^N_+) \times \dot H^1_q(\R^N_+)}
+\|\nabla \fp\|_{L_q(\R^N_+)}
\le C \|(\bff, \nabla \bG)\|_{L_q(\R^N_+)}.
\end{aligned}
\end{equation}
\end{cor}

Finally, let us recall the Weis operator-valued Fourier multiplier theorem, which is one of the tools to obtain the maximal regularity.
Let $\CD(\BR,X)$ be the set of all $X$ 
valued $C^{\infty}$ functions having compact support,
Given $M \in L_{1,\rm{loc}}(\BR \backslash \{0\}, \CL(X, Y))$, 
we define the operator $T_{M} : \CF^{-1} \CD(\BR,X)\rightarrow \CS'(\BR,Y)$ 
by
\begin{align}\label{eqTM}
T_M \phi=\CF^{-1}[M\CF[\phi]],\quad (\CF[\phi] \in \CD(\BR,X)). 
\end{align}

\begin{thm}[Weis \cite{W}]\label{Weis}
Let $X$ and $Y$ be two UMD Banach spaces and $1 < p < \infty$. 
Let $M$ be a function in $C^{1}(\BR \backslash \{0\}, \CL(X,Y))$ such that 
\begin{align*}
\CR_{\CL(X,Y)} (\{(\zeta \frac{d}{d\zeta})^{\ell} M(\zeta) \mid
 \zeta \in \BR \backslash \{0\}\}) \leq m < \infty
\quad (\ell =0,1)
\end{align*}
with some constant $m$. 
Then the operator $T_{M}$ defined in \eqref{eqTM} 
is extended to a bounded linear operator from
 $L_{p}(\BR,X)$ into $L_{p}(\BR,Y)$. 
Moreover, denoting this extension by $T_{M}$, we have 
\begin{align*}
	\|T_{M}\|_{\CL(L_p(\BR,X),L_p(\BR,Y))} \leq Cm
\end{align*}
for some positive constant $C$ depending on $p$, $X$ and $Y$. 
\end{thm}

\section{Maximal regularity}\label{sec:mr}
In this section, we prove the maximal $L_p$-$L_q$ regularity for the following linearized problem: 
\begin{equation}\label{linear}
\left\{
\begin{aligned}
	&\pd_t\bu -\Delta \bu + \nabla \fp + \beta \DV (\Delta \bQ -a \bQ)=\bff,
	\enskip \dv \bu=0& \quad&\text{in $\R^N_+$}, \enskip t \in \R_+,\\
	&\pd_t \bQ - \beta \bD(\bu) - \Delta \bQ + a \bQ =\bG& \quad&\text{in $\R^N_+$}, \enskip t \in \R_+,\\
	&\bu = 0, \enskip \pd_N \bQ = 0& \quad&\text{on $\R^N_0$}, \enskip t \in \R_+,\\
	&(\bu, \bQ)|_{t=0}=(\bu_0, \bQ_0)& \quad&\text{in $\R^N_+$}.
\end{aligned}
\right.
\end{equation}
Let us state the main result in this section. 
\begin{thm}\label{thm:local mr}
Let $N \ge 2$.
Let $1<p, q<\infty$.
For any 
\[
\begin{aligned}
	&\bff \in L_p(\R_+, L_q(\R^N_+)^N),\quad
	\bG \in L_p(\R_+, \dot H^1_q(\R^N_+; \bS_0))
\end{aligned}
\]
and $(\bu_0, \bQ_0) \in \CD_{q, p}(\R^N_+)$,
the linearized problem
\eqref{linear} admits a unique solution
$(\bu, \bQ, \fp)$ with
\begin{align*}
	\pd_t \bu &\in L_p (\R_+, L_q(\R^N_+)^N), & 
	\bu &\in L_p (\R_+, \dot H^2_q(\R^N_+)^N),\\
	\pd_t \bQ &\in L_p (\R_+, \dot H^1_q(\R^N_+; \bS_0)), &
	\bQ &\in L_p (\R_+, \dot H^1_q(\R^N_+; \bS_0) \cap \dot H^3_q(\R^N_+; \bS_0)),\\
	\nabla \fp &\in L_p (\R_+, L_q(\R^N_+)^N) & &
\end{align*}
possessing the estimate:
\begin{equation}\label{local}
\begin{aligned}
		&\|(\pd_t, \nabla^2) (\bu, \bQ)\|_{L_p(\R_+, \dot H^{0, 1}_q(\R^N_+))}
		+\|\nabla \bQ\|_{L_p(\R_+, L_q(\R^N_+))}
		+\|\nabla \fp\|_{L_p(\R_+, L_q(\R^N_+))}\\
		&\enskip \le C (\|(\bu_0, \bQ_0)\|_{\CD_{q, p}(\R^N_+)} + \|(\bff, \nabla \bG)\|_{L_p(\R_+, L_q(\R^N_+))})
\end{aligned}
\end{equation}
with some positive constant $C$.
\end{thm}

In order to show Theorem \ref{thm:local mr}, we
first consider 
\begin{equation}\label{eq:extend}
\left\{
\begin{aligned}
&\pd_t\bu -\Delta \bu + \nabla \fp + \beta \DV (\Delta \bQ -a \bQ)=\bff,
\enskip \dv \bu=0& \quad&\text{in $\R^N_+$}, \enskip t \in \R,\\
&\pd_t \bQ - \beta \bD(\bu) - \Delta \bQ + a \bQ =\bG& \quad&\text{in $\R^N_+$}, \enskip t \in \R,\\
&\bu=0, \enskip \pd_N \bQ=0 & \quad&\text{on $\R^N_0$}, \enskip t \in \R.
\end{aligned}
\right.
\end{equation}
Let
\[
	\bF(t)=(\bff(t), \nabla \bG(t)). 
\]
Thanks to Theorem \ref{thm:Rbdd H}, the solution $(\bu, \bQ)$ of \eqref{eq:extend} are written by
\begin{align*}
	(\pd_t, \nabla^2)\bu(\cdot, t)&=\CL^{-1}[(\lambda, \nabla^2)\CA(\lambda)\CL[\bF]](t)=\CF^{-1}[(\lambda, \nabla^2)\CA(\lambda)\CF[\bF]](t),
	\\
	(\pd_t \nabla , \nabla^3)\bQ(\cdot, t)&=\CL^{-1}[(\lambda \nabla, \nabla^3)\CB(\lambda)\CL[\bF]](t)=\CF^{-1}[(\lambda \nabla, \nabla^3)\CB(\lambda)\CF[\bF]](t)
\end{align*}
for $\lambda =i\tau \in i\R \setminus\{0\}$, which implies that
we are ready to apply Theorem \ref{Weis}.
Then we have
\[
	\|(\pd_t, \nabla^2) (\bu, \bQ)\|_{L_p(\R, \dot H^{0, 1}_q(\R^N_+))}
	\le C \|\bF\|_{L_p(\R, L_q(\R^N_+))}.
\]
Furthermore, the second equation of \eqref{eq:extend} yields that
\begin{align*}
	&\|\nabla \bQ\|_{L_p(\R, L_q(\R^N_+))}\\
	&\le C(\|\pd_t \nabla \bQ\|_{L_p(\R, L_q(\R^N_+))}+\|\nabla^2 \bu\|_{L_p(\R, L_q(\R^N_+))}+\|\nabla^3 \bQ\|_{L_p(\R, L_q(\R^N_+))}+\|\nabla \bG\|_{L_p(\R, L_q(\R^N_+))})\\
	&\le C \|\bF\|_{L_p(\R, L_q(\R^N_+))}.
\end{align*}
In the following, we consider the existence of the pressure term.
Let $(\bu, \bQ)$ be a solution of \eqref{eq:extend} for $\bF \in L_p(\R, \dot H^{0, 1}_q(\R^N_+))$.
The weak Dirichlet Neumann problem
\begin{align*}
	(\nabla p_1, \nabla \vp)&=(\Delta \bu - \beta \DV(\Delta \bQ-a\bQ), \nabla \vp)\\
	(\nabla p_2, \nabla \vp)&=(\bff, \nabla \vp) 
\end{align*}
have a unique solution $p_1(t)=K_1(\bu(t), \bQ(t)) \in \widehat H^1_{q, 0}(\R^N_+)$, $p_2(t)=K_2(\bff(t)) \in \widehat H^1_{q, 0}(\R^N_+)$ for any $\vp \in \widehat H^1_{q', 0}(\R^N_+)$, respectively, then setting
$\fp=K_1(\bu(t), \bQ(t))+K_2(\bff(t))$, $\fp$ is a solution of \eqref{eq:extend}
with
\begin{align*}
	\|\nabla \fp\|_{L_p(\R, L_q(\R^N_+))}&\le C(\|\Delta \bu - \beta \DV(\Delta \bQ-a\bQ)\|_{L_p(\R, L_q(\R^N_+))}
	+\|\bff\|_{L_p(\R, L_q(\R^N_+))})\\
	&\le C\|\bF\|_{L_p(\R, L_q(\R^N_+))}.
\end{align*}
Then we have the following lemma.

\begin{lem}\label{lem:extend}
Let $1<p, q<\infty$. For any $\bff$ and $\bG$ with
\[
	\bff \in L_p(\R, L_q(\R^N_+)^N),
	\quad \bG \in L_p(\R, \dot H^1_q(\R^N_+; \bS_0)),
\]
\eqref{eq:extend} admits 
a solution $(\bu, \bQ, \fp)$ with
\begin{equation}\label{regularity:ext}
\begin{aligned}
	\pd_t \bu &\in L_p (\R_+, L_q(\R^N_+)^N), & 
	\bu &\in L_p (\R_+, \dot H^2_q(\R^N_+)^N),\\
	\pd_t \bQ &\in L_p (\R_+, \dot H^1_q(\R^N_+; \bS_0)), &
	\bQ &\in L_p (\R_+, \dot H^1_q(\R^N_+; \bS_0) \cap \dot H^3_q(\R^N_+; \bS_0)),\\
	\nabla \fp &\in L_p (\R_+, L_q(\R^N_+)^N) & &
\end{aligned}
\end{equation}
possessing the estimate 
\begin{align*}
	&\|(\pd_t, \nabla^2) (\bu, \bQ)\|_{L_p(\R, \dot H^{0, 1}_q(\R^N_+))}
	+\|\nabla \bQ\|_{L_p(\R, L_q(\R^N_+))}
	+\|\nabla \fp\|_{L_p(\R, L_q(\R^N_+))}
	\le C \|(\bff, \nabla \bG)\|_{L_p(\R, L_q(\R^N_+))}. 
\end{align*}
\end{lem}

Second, we consider the following linearized problem in the semigroup setting.
\begin{equation}\label{eq:initial}
\left\{
\begin{aligned}
	&\pd_t\bu -\Delta \bu + \nabla \fp + \beta \DV (\Delta \bQ -a \bQ)=0,
	\enskip \dv \bu=0& \quad&\text{in $\R^N_+$}, \enskip t \in \R_+,\\
	&\pd_t \bQ - \beta \bD(\bu) - \Delta \bQ + a \bQ =0& \quad&\text{in $\R^N_+$}, \enskip t \in \R_+,\\
	&\bu=0, \enskip \pd_N \bQ=0 & \quad&\text{on $\R^N_0$}, \enskip t \in \R_+,\\
	&(\bu, \bQ)|_{t=0}=(\bu_0, \bQ_0)& \quad&\text{in $\R^N_+$}.
\end{aligned}
\right.
\end{equation}
Let us consider the resolvent problem
corresponding to \eqref{eq:initial}:
\begin{equation}\label{resolvent prob U2}
\left\{
\begin{aligned}
	&\lambda\bu -\Delta \bu + \nabla \fp + \beta \DV (\Delta \bQ -a \bQ)=\bff,
	\enskip \dv \bu=0& \quad&\text{in $\R^N_+$},\\
	&\lambda \bQ - \beta \bD(\bu) - \Delta \bQ + a \bQ =\bG& \quad&\text{in $\R^N_+$},\\
	&\bu=0, \enskip \pd_N \bQ=0 & \quad&\text{on $\R^N_0$}.
\end{aligned}
\right.
\end{equation}
For any $\bu \in \dot H^2_q(\R^N_+)^N$ and $\bQ \in \dot H^3_q(\R^N_+; \bS_0) \cap \dot H^1_q(\R^N_+; \bS_0)$,
let $p=K(\bu, \bQ) \in \widehat H^1_{q, 0}(\R^N_+)$ be a solution of the weak Dirichlet Neumann problem:
\[
	(\nabla p, \nabla \vp)=(\Delta \bu-\beta \DV(\Delta \bQ-a\bQ), \nabla \vp)
\]
for any $\vp \in \widehat H^1_{q', 0}(\R^N_+)$ satisfying
\[
	\|\nabla K(\bu, \bQ)\|_{L_q(\R^N_+)} \le C(\|\bu\|_{\dot H^2_q(\R^N_+)}+\|\bQ\|_{\dot H^1_q(\R^N_+)}+\|\bQ\|_{\dot H^3_q(\R^N_+)}).
\]
Then we introduce the reduced problem.
\begin{equation}\label{reduced}
\left\{
\begin{aligned}
	&\lambda\bu -\Delta \bu + \nabla K(\bu, \bQ) + \beta \DV (\Delta \bQ -a \bQ)=\bff& \quad&\text{in $\R^N_+$},\\
	&\lambda \bQ - \beta \bD(\bu) - \Delta \bQ + a \bQ =\bG& \quad&\text{in $\R^N_+$},\\
	&\bu=0, \enskip \pd_N \bQ=0 & \quad&\text{on $\R^N_0$}.
\end{aligned}
\right.
\end{equation}
If $\bff \in J_q(\R^N_+)$, the existence of a solution $(\bu, \bQ, \fp) \in H^2_q(\R^N_+)^N \times H^3_q(\R^N_+; \bS_0) \times \widehat H^1_{q, 0}(\R^N_+)$ to \eqref{resolvent prob U2} is equivalent to the existence of a solution $(\bu, \bQ) \in H^2_q(\R^N_+)^N \times H^3_q(\R^N_+; \bS_0)$ to \eqref{reduced}.
In particular, if $(\bu, \bQ) \in H^2_q(\R^N_+)^N \times H^3_q(\R^N_+; \bS_0)$ is a solution to \eqref{reduced}, we have $\bu \in J_q(\R^N_+)$.
Hence, we have $\dv \bu =0$ in the sense of distributions.
Recall the definitions of $\CD(\CA_q)$ and $\CA_q(\bu, \bQ)$, together with \eqref{homo Sobolev}, that
\begin{align*}
	\CD(\CA_q)&
	=\{(\bu, \bQ)\in (H^2_q(\R^N_+)^N \cap J_q(\R^N_+)) \times (\dot H^3_q(\R^N_+; \bS_0) \cap \dot H^1_q(\R^N_+; \bS_0)) \mid \bu|_{x_N=0} = 0, \enskip \pd_N \bQ|_{x_N=0}=0\},\\
	\CA_q(\bu, \bQ)&=(\Delta \bu-\nabla K(\bu, \bQ)-\beta \DV(\Delta \bQ-a\bQ), \beta \bD(\bu) + \Delta \bQ - a\bQ) \enskip \text{for} \enskip (\bu, \bQ) \in \CD(\CA_q).
\end{align*}
The resolvent estimate \eqref{rem:resolvent} implies that $\CA_q$ generates an analytic semigroup $\{T(t)\}_{t\ge 0}$ on $\CX_q(\R^N_+) = J_q(\R^N_+) \times \dot H^1_q(\R^N_+; \bS_0)$ 
with $\|(\bff, \bG)\|_{\CX_q(\R^N_+)}=\|(\bff, \bG)\|_{\dot H^{0, 1}_q(\R^N_+)}$. 
Furthermore, the following estimates follow from \eqref{rem:resolvent} and standard analytic semigroup arguments.
\begin{align}
	\|\pd_t T(t)(\bff, \bG)\|_{\CX_q(\R^N_+)}
	&\le Ct^{-1}\|(\bff, \bG)\|_{\CX_q(\R^N_+)}& \enskip&\text{for $(\bff, \bG)\in \CX_q(\R^N_+)$}, \label{est:semi1}\\
	\|\pd_t T(t)(\bff, \bG)\|_{\CX_q(\R^N_+)}
	&\le C\|\CA_q(\bff, \bG)\|_{\CX_q(\R^N_+)} 
	\le C \|(\bff, \bG)\|_{\CD(\CA_q)}& \enskip&\text{for $(\bff, \bG)\in \CD(\CA_q)$}.\label{est:semi2}
\end{align}
Here, $\|\cdot\|_{\CD(\CA_q)}$ denotes the graph norm of $\CA_q$.
In addition, it follows from the same method as the proof of \cite[Proposition 4.9 (1)]{DHMT} that $\|\CA_q(\bff, \bG)\|_{\CX_q(\R^N_+)}$ coincides with $\|\nabla^2(\bff, \bG)\|_{\CX_q(\R^N_+)}$.
Thus, we may write
\[
	\|(\bff, \bG)\|_{\CD(\CA_q)} = \|(\bff, \bG)\|_{H^2_q(\R^N_+) \times (\dot H^3_q(\R^N_+) \cap \dot H^1_q(\R^N_+))}.
\]
Recall that
\[
	\CD_{q, p}(\R^N_+)=(\CX_q(\R^N_+), \CD(\CA_q))_{1-1/p, p}.
\]
It holds by \eqref{est:semi1} and \eqref{est:semi2} with a real interpolation method that
\[
	\|\pd_t T(t)(\bff, \bG)\|_{L_p(\R_+, \CX_q(\R^N_+))}
	 \le C \|(\bff, \bG)\|_{\CD_{q, p}(\R^N_+)}
\]
for $(\bff, \bG) \in \CD_{q, p}(\R^N_+)$, where we refer \cite[Proof of Theorem 3.9]{SS2} for the details.
Since $\pd_t T(t) = \CA_q T(t)$, we have
\begin{equation}\label{est:semi3}
\begin{aligned}
	&\|\pd_t T(t)(\bff, \bG)\|_{L_p(\R_+, \CX_q(\R^N_+))}+\|\nabla^2 T(t)(\bff, \bG)\|_{L_p(\R_+, \CX_q(\R^N_+))}\\
	&\le C \|(\bff, \bG)\|_{\CD_{q, p}(\R^N_+)}.
\end{aligned}
\end{equation}
Therefore, setting $(\bu(t), \bQ(t)) = T(t)(\bu_0, \bQ_0)$ and $\fp=K(\bu(t), \bQ(t))$ for $(\bu_0, \bQ_0) \in \CD_{q, p}(\R^N_+)$,
$(\bu, \bQ, \fp)$ is a unique solution of \eqref{eq:initial} such that
$(\pd_t \bu, \pd_t \bQ) \in L_p(\R_+, \CX_q(\R^N_+))$, $(\bu, \bQ) \in L_p(\R_+, \CD(\CA_q))$,
and $\nabla \fp \in L_p(\R_+, L_q(\R^N_+)^N)$
with
\begin{align*}
	\|\pd_t (\bu, \bQ)\|_{L_p(\R_+, \CX_q(\R^N_+))}+\|\nabla^2 (\bu, \bQ)\|_{L_p(\R_+, \CX_q(\R^N_+))} + \|\nabla \fp\|_{L_p(\R_+, L_q(\R^N_+))}
	\le C \|(\bu_0, \bQ_0)\|_{\CD_{q, p}(\R^N_+)}.
\end{align*}
Furthermore, the second equation of \eqref{eq:initial} implies that 
\begin{align*}
	\|\nabla \bQ\|_{L_p(\R_+, L_q(\R^N_+))} 
	&\le C(\|\pd_t \nabla \bQ\|_{L_p(\R_+, L_q(\R^N_+))}+\|\nabla^2 \bu\|_{L_p(\R_+, L_q(\R^N_+))}+\|\nabla^3 \bQ\|_{L_p(\R_+, L_q(\R^N_+))})\\
	&\le C \|(\bu_0, \bQ_0)\|_{\CD_{q, p}(\R^N_+)}.
\end{align*}
Therefore, we have the following lemma.
\begin{lem}\label{lem:initial}
Let $N \ge 2$.
Let $1<p, q<\infty$.
For any 
$(\bu_0, \bQ_0) \in \CD_{q, p}(\R^N_+)$,
the linearized problem
\eqref{eq:initial} admits 
a solution $(\bu, \bQ, \fp)$ with
\begin{equation}\label{regularity:initial}
\begin{aligned}
	\pd_t \bu &\in L_p (\R_+, L_q(\R^N_+)^N), & 
	\bu &\in L_p (\R_+, \dot H^2_q(\R^N_+)^N),\\
	\pd_t \bQ &\in L_p (\R_+, \dot H^1_q(\R^N_+; \bS_0)), &
	\bQ &\in L_p (\R_+, \dot H^1_q(\R^N_+; \bS_0) \cap \dot H^3_q(\R^N_+; \bS_0)),\\
	\nabla \fp &\in L_p (\R_+, L_q(\R^N_+)^N) & &
\end{aligned}
\end{equation}
possessing the estimate 
\begin{align*}
	\|(\pd_t, \nabla^2) (\bu, \bQ)\|_{L_p(\R_+, \dot H^{0, 1}_q(\R^N_+))}
	+\|\nabla \bQ\|_{L_p(\R_+, L_q(\R^N_+))}
	+\|\nabla \fp\|_{L_p(\R_+, L_q(\R^N_+))}
	\le C \|(\bu_0, \bQ_0)\|_{\CD_{q, p}(\R^N_+)}.
\end{align*}
\end{lem}

Lemma \ref{lem:extend} and Lemma \ref{lem:initial} furnish the maximal regularity for \eqref{linear}.

\begin{proof}[Proof of Theorem \ref{thm:local mr}]
Let
$e_T[f]$ be a zero extension of $f$; namely,
\begin{equation}\label{extension}
	e_T[f]=
\left\{
\begin{aligned}
	&0& \quad&t<0,\\
	&f(t)& \quad&t>0.
\end{aligned}
\right.
\end{equation}
Let $\bU_j=(\bu_j, \bQ_j)$ for $j=1,2$.
Assume that $(\bU_1, \fp_1)$ and $(\bU_2, \fp_2)$ satisfy the following problems:
\begin{equation}\label{eq:U1}
\left\{
\begin{aligned}
	&\pd_t\bu_1 -\Delta \bu_1 + \nabla \fp_1 + \beta \DV (\Delta \bQ_1 -a \bQ_1)=e_T[\bff],
	\enskip \dv \bu_1=0& \quad&\text{in $\R^N_+$}, \enskip t \in \R,\\
	&\pd_t \bQ_1 - \beta \bD(\bu_1) - \Delta \bQ_1 + a \bQ_1 =e_T[\bG]& \quad&\text{in $\R^N_+$}, \enskip t \in \R,\\
	&\bu_1 = 0, \enskip \pd_N \bQ_1 = 0& \quad&\text{on $\R^N_0$}, \enskip t \in \R.
\end{aligned}
\right.
\end{equation}

\begin{equation}\label{eq:U2}
\left\{
\begin{aligned}
	&\pd_t\bu_2 -\Delta \bu_2 + \nabla \fp_2 + \beta \DV (\Delta \bQ_2 -a \bQ_2)=0,
	\enskip \dv \bu_2=0& \quad&\text{in $\R^N_+$}, \enskip t \in \R_+,\\
	&\pd_t \bQ_2 - \beta \bD(\bu_2) - \Delta \bQ_2 + a \bQ_2 =0& \quad&\text{in $\R^N_+$}, \enskip t \in \R_+,\\
	&\bu_2 = 0, \enskip \pd_N \bQ_2 = 0,& \quad&\text{on $\R^N_0$}, \enskip t \in \R_+,\\
	&(\bu_2, \bQ_2)|_{t=0}=(\bu_0-\bu_1(0), \bQ_0-\bQ_1(0)) & \quad&\text{in $\R^N_+$}.
\end{aligned}
\right.
\end{equation}
Then $\bU=\bU_1+\bU_2$ and $\fp=\fp_1+\fp_2$ satisfy \eqref{linear} for $t\in \R_+$.
In the following, we consider the estimates of $\bU_1$ and $\bU_2$.

First, we consider \eqref{eq:U1}.
Let 
$\widetilde \bF(t)=(e_T[\bff], \nabla e_T[\bG])$.
Then Lemma \ref{lem:extend}, together with  
\begin{equation*}\label{est:F}
	\|\widetilde \bF\|_{L_p(\R, L_q(\R^N_+))}
	\le C\|(\bff, \nabla \bG)\|_{L_p(\R_+, L_q(\R^N_+))},
\end{equation*} 
furnishes that there exists $(\bU_1, \fp_1)$ satisfying the regularity conditions 
\[
\begin{aligned}
	\pd_t \bu_1 &\in L_p (\R_+, L_q(\R^N_+)^N), & 
	\bu_1 &\in L_p (\R_+, \dot H^2_q(\R^N_+)^N),\\
	\pd_t \bQ_1 &\in L_p (\R_+, \dot H^1_q(\R^N_+; \bS_0)), &
	\bQ_1 &\in L_p (\R_+, \dot H^1_q(\R^N_+; \bS_0) \cap \dot H^3_q(\R^N_+; \bS_0)),\\
	\nabla \fp_1 &\in L_p (\R_+, L_q(\R^N_+)^N) & &
\end{aligned}
\]
and 
\begin{equation}\label{est:U1}
\begin{aligned}
	&\|(\pd_t, \nabla^2) \bU_1\|_{L_p(\R_+, \dot H^{0, 1}_q(\R^N_+))}
	+\|\nabla \bQ_1\|_{L_p(\R_+, L_q(\R^N_+))}
	+\|\nabla \fp_1\|_{L_p(\R_+, L_q(\R^N_+))}\\
	&\le C\|\widetilde \bF\|_{L_p(\R, L_q(\R^N_+))} 
	\le C\|(\bff, \nabla \bG)\|_{L_p(\R_+, L_q(\R^N_+))}.
\end{aligned}
\end{equation}

Next, we consider \eqref{eq:U2}. Let $\bU_0 = (\bu_0, \bQ_0) \in \CD_{q, p}(\R^N_+)$.
To apply Lemma \ref{lem:initial}, we verify the initial data for \eqref{eq:U2} belongs to $\CD_{q, p}(\R^N_+)$.
To achieve that,
let us prove $\bU_1(0)=0$.
First, we represent the solution formula of \eqref{eq:U1}.
Applying the Laplace transform to \eqref{eq:U1},
we have the resolvent problem:
\begin{equation}\label{resolvent eq:U1}
\left\{
\begin{aligned}
	&\lambda \widehat \bu_1 -\Delta \widehat \bu_1 + \nabla \widehat \fp_1 + \beta \DV (\Delta \widehat \bQ_1 -a \widehat \bQ_1)=\CL[e_T[\bff]],
	\enskip \dv \widehat \bu_1=0& \quad&\text{in $\R^N_+$},\\
	&\lambda \widehat \bQ_1 - \beta \bD(\widehat \bu_1) - \Delta \widehat \bQ_1 + a \widehat \bQ_1 =\CL[e_T[\bG]]& \quad&\text{in $\R^N_+$},\\
	&\widehat \bu_1 = 0, \enskip \pd_N \widehat \bQ_1=0 & \quad&\text{on $\R^N_0$},
\end{aligned}
\right.
\end{equation}
where we have set $\CL[f]=\widehat f$.
Theorem \ref{thm:Rbdd H} implies that 
$\widehat \bu_1=\CA(\lambda) \CL[\widetilde \bF]$ and
$\widehat \bQ_1=\CB(\lambda) \CL[\widetilde \bF]$
satisfy \eqref{resolvent eq:U1} for $\lambda \in \Sigma_\epsilon$. 
Thus, we may write solution formulas for $(\bu_1, \bQ_1)$ of \eqref{eq:U1} as follows:
\[
	\bu_1=\CL^{-1}[\CA(\lambda) \CL[\widetilde \bF]],  \enskip
	\bQ_1=\CL^{-1}[\CB(\lambda) \CL[\widetilde \bF]].
\]
Let $\gamma_0>0$.
Since $\CL[\widetilde \bF]$, $\CA(\lambda)$, and $\CB(\lambda)$ are holomorphic for $\Re \lambda \ge \gamma_0$, the Cauchy's theorem and the Fubini's theorem furnish that
\begin{equation}\label{formula:u1}
\begin{aligned}
	\bu_1&=\CL^{-1}[\CA(\lambda) \CL[\widetilde \bF]]\\
	&=\frac{1}{2\pi} \int^\infty_{-\infty} e^{(\gamma_0 + i \tau) t}\CA(\gamma_0 + i \tau) \int^\infty_{-\infty} e^{-(\gamma_0 + i \tau) s} \widetilde \bF(s)\,ds\,d\tau\\
	&=\frac{1}{2\pi} \int^\infty_{-\infty} \int^\infty_{-\infty} e^{(\gamma_0 + i \tau) (t-s)}\CA(\gamma_0 + i \tau) \widetilde \bF(s)\,d\tau\,ds,
\end{aligned}
\end{equation}
and also
\begin{equation}\label{formula:q1}
	\bQ_1=\CL^{-1}[\CB(\gamma_0 + i \tau) \CL[\widetilde \bF]]
	=\frac{1}{2\pi} \int^\infty_{-\infty} \int^\infty_{-\infty} e^{(\gamma_0 + i \tau) (t-s)}\CB(\gamma_0 + i \tau) \widetilde \bF(s)\,d\tau\,ds.
\end{equation}
Let
\begin{equation}\label{semi'}
	S(t) \widetilde \bF = 
\left\{
\begin{aligned}
	&\frac{1}{2\pi} \int^\infty_{-\infty} e^{(\gamma_0 + i \tau) t} (\CA(\gamma_0 + i \tau) \widetilde \bF, \CB(\gamma_0 + i \tau) \widetilde \bF)\,d\tau &\quad &\text{for $t\neq 0$},\\
	&\widetilde \bF &\quad &\text{for $t=0$}.
\end{aligned}
\right.
\end{equation}
By \eqref{formula:u1} and \eqref{formula:q1},
we may write
\begin{equation}\label{formula:U1}
\bU_1(t) =\int^\infty_{-\infty} S(t-s) \widetilde \bF(s)\,ds
\end{equation}
for $t \neq 0$.
Let $\Gamma_\omega = \Gamma_\omega^+ \cup \Gamma_\omega^- \cup C_\omega$
for $\omega >0$, where
\begin{equation}\label{int path}
\begin{aligned}
	\Gamma_\omega^\pm &= \{\lambda = re^{\pm i(\pi - \epsilon)} \mid \omega < r < \infty\},\\
	C_\omega &= \{\lambda = \omega e^{i\eta} \mid -(\pi - \epsilon) < \eta < (\pi - \epsilon)\}.
\end{aligned}
\end{equation}
By the same calculation as \cite[proof of Theorem 5.1]{SW},
we have
	\begin{numcases}
{S(t) \widetilde \bF =}	 
	0 &\enskip \text{for $t<0$}, \label{semi negative}\\
	\frac{1}{2\pi i} \int_{\Gamma_\omega} e^{\lambda t}
	(\CA(\lambda) \widetilde \bF, \CB(\lambda) \widetilde \bF)\,d\lambda
	&\enskip \text{for $t>0$}, \label{semi}\\
	\widetilde \bF &\enskip \text{for $t=0$} \label{semi t0}.
	\end{numcases}
It holds by \eqref{rem:resolvent} that
\[
	\|(\CA(\lambda) \widetilde \bF, \CB(\lambda) \widetilde \bF)\|_{\dot H^{0, 1}_q(\R^N_+)} 
	\le C|\lambda|^{-1}\|\widetilde \bF\|_{L_q(\R^N_+)}
\]
for $\lambda \in \Sigma_\epsilon$.  
Then according to the argument in the theory of an analytic semigroup, 
\eqref{semi} and \eqref{semi t0} 
imply that $\{S(t)\}_{t \ge 0}$ is analytic semigroup generated by $\CA_q$.
In particular, there exists a constant $M>0$ such that
\begin{equation}\label{est:semi}
	\|S(t) \widetilde \bF\|_{\dot H^{0, 1}_q(\R^N_+)} \le M\|\widetilde \bF\|_{L_q(\R^N_+)}.
\end{equation}
Set
\begin{align*}
	V_1 &= \{\bu \mid \bu \in L_p(\R_+, \dot H^2_q(\R^N_+)), \pd_t \bu \in L_p(\R_+, L_q(\R^N_+))\},\\
	V_2 &=\{\bQ \mid \bQ \in L_p(\R_+, \dot H^3_q(\R^N_+; \bS_0) \cap \dot H^1_q(\R^N_+; \bS_0)), \pd_t \bQ \in L_p(\R_+, \dot H^1_q(\R^N_+; \bS_0))\}.
\end{align*}
The embedding property \cite[(1.17)]{T}, together with \eqref{interpolation Q}, furnishes that
\begin{align*}
	&V_1 \subset C([0, \infty), \dot B^{2(1-1/p)}_{q, p}(\R^N_+)),\\
	&V_2
	\subset C([0, \infty), \dot H^1_q(\R^N_+; \bS_0) \cap \dot B^{3-2/p}_{q, p}(\R^N_+; \bS_0))
\end{align*}
for $1<p, q<\infty$.
Since $\bU_1$ satisfies \eqref{regularity:ext}, we have $\bU_1 \in V_1 \times V_2$; therefore,
$\bU_1(t)$ is continuous at $t=0$. 
Furthermore, by \eqref{formula:U1} and \eqref{semi negative}, we may write
\[
\bU_1(t) =
\left\{
	\begin{aligned}
	& \int^t_{-\infty} S(t-s) \widetilde \bF(s)\,ds&\quad &\text{for $t\neq 0$},\\
	&\lim_{t\to 0} \int^t_{-\infty} S(t-s) \widetilde \bF(s)\,ds&\quad &\text{for $t=0$}.
\end{aligned}
\right.
\]
Here, we prove  
\[
	\int^t_{-\infty} S(t-s) \widetilde \bF(s)\,ds
\]
is continuous at $t=0$.
Since  
\begin{equation}\label{conti0}
\int^t_{-\infty} S(t-s) \widetilde \bF(s)\,ds = \int^\infty_0 S(\ell) \widetilde \bF(t-\ell)\,d\ell,
\end{equation}
we prove that the right-hand side of \eqref{conti0} is continuous at $t=0$.
It follows from \eqref{est:semi} that
\begin{equation}\label{conti1}
	\|\int^\infty_0 S(\ell) (\widetilde \bF(t-\ell) - \widetilde \bF(-\ell))\,d\ell\|_{\dot H^{0, 1}_q(\R^N_+)} \le M \int^\infty_0 \|\widetilde \bF(t-\ell) - \widetilde \bF(-\ell)\|_{L_q(\R^N_+)} \,d\ell.
\end{equation}
Set $\delta_0 = M^{-1}$.
The definition of $\widetilde \bF$ implies that
$\widetilde \bF(t-\ell) = 0$ for $|t|<\delta_0$ if $\ell>\delta_0$.
Therefore, for $|t| < \delta_0$, we have
\begin{equation}\label{conti2}
	\int^\infty_0 \|\widetilde \bF(t-\ell) - \widetilde \bF(-\ell)\|_{L_q(\R^N_+)} \,d\ell 
	= \int^{\delta_0}_0 \|\widetilde \bF(t-\ell) - \widetilde \bF(-\ell)\|_{L_q(\R^N_+)} \,d\ell.
\end{equation}
Furthermore, since $C^\infty_0(\R, L_q(\R^N_+))$ is dense in $L_p(\R, L_q(\R^N_+))$, for any $\epsilon$, there exists $\delta_1>0$ such that for all $t \in \R$ with $|t|<\delta_1$, $\widetilde \bF$ satisfies 
\begin{equation}\label{conti3}
	\|\widetilde \bF(t) - \widetilde \bF(0)\|_{L_q(\R^N_+)} < \epsilon.
\end{equation}
Combining \eqref{conti1}, \eqref{conti2}, and \eqref{conti3}, 
for $|t| < \delta:=\min\{\delta_0, \delta_1\}$, we have
\[
	\|\int^\infty_0 S(\ell) (\widetilde \bF(t-\ell) - \widetilde \bF(-\ell))\,d\ell\|_{\dot H^{0, 1}_q(\R^N_+)} < \epsilon,
\]
which implies that the right-hand side of \eqref{conti0} is continuous at $t=0$.
Therefore, we may write
\begin{equation*}\label{U1}
\bU_1(t) =\int^t_{-\infty} S(t-s) \widetilde \bF(s)\,ds
\end{equation*}
for any $t\in \R$.
In particular, we have
\[
	\bU_1(0) =\int^0_{-\infty} S(-s) \widetilde \bF(s)\,ds.
\] 
Then it holds by \eqref{est:semi} that
\begin{equation}\label{U1(0)}
	\|\bU_1(0)\|_{\dot H^{0, 1}_q(\R^N_+)} \le M \int^0_{-\infty} \|\widetilde \bF(s)\|_{L_q(\R^N_+)} \,ds = 0,
\end{equation}
which implies that $\bU_1(0)=0$.
Therefore, we can apply Lemma \ref{lem:initial} to \eqref{eq:U2}, then it holds that
there exists $(\bU_2, \fp_2)$ satisfying the regularity conditions \eqref{regularity:initial}
and 
\begin{equation*}\label{est:initial}
\begin{aligned}
	\|(\pd_t, \nabla^2)\bU_2\|_{L_p(\R_+, \dot H^{0, 1}_q(\R^N_+))}
	+\|\nabla \bQ_2\|_{L_p(\R_+, L_q(\R^N_+))}
	+\|\nabla \fp_2\|_{L_p(\R_+, L_q(\R^N_+))}
	 \le C\|(\bu_0, \bQ_0)\|_{\CD_{q, p}(\R^N_+)},
\end{aligned}
\end{equation*}
together with \eqref{est:U1}, we have \eqref{local}.

Finally, we mention the uniqueness of the solutions.
Let us consider the homogeneous equation:
\begin{equation}\label{homo U}
	\pd_t \bU-\CA_q \bU=0 \enskip \text{in $\R^N_+$, \enskip $t \in \R_+$},
	\quad
	\bU|_{t=0}=0
\end{equation}
with
\begin{equation}\label{space U}
	\pd_t \bU \in L_p(\R_+, \CX_q(\R^N_+)), \enskip \bU \in L_p(\R_+, \CD(\CA_q)).
\end{equation}
Let $\bV$ be the zero extension of $\bU$ to $t<0$.
Then \eqref{homo U} implies that $\bV$ satisfies  
\[
	\pd_t \bV-\CA_q \bV=0 \enskip \text{in $\R^N_+$, \enskip $t \in \R$}.
\]
For any $\lambda \in \C$ with $\Re \lambda=\gamma > 0$, 
we set 
\[
	\widehat \bU(\lambda) = \int^\infty_{-\infty} e^{-\lambda t} \bV(t)\,dt
	=\int^\infty_0 e^{-\lambda t} \bU(t)\,dt.
\]
H\"older inequality and \eqref{space U} implies that
\begin{align*}
	\|\widehat \bU(\lambda)\|_{\CD(\CA_q)}
	&\le \left( \int^\infty_0 e^{-\gamma tp'}\,dt \right)^{1/p'}
	\|\bU\|_{L_p(\R_+, \CD(\CA_q))}\\
	&=(\gamma p')^{-1/p'} \|\bU\|_{L_p(\R_+, \CD(\CA_q))}.
\end{align*}
Note that $\lambda \widehat \bU$ is also meaningful in $\dot H^{0, 1}_q(\R^N_+)$. 
In fact, 
since $\lambda \widehat \bU = \int^\infty_0 e^{-\lambda t}\pd_t \bU\,dt$,
we have
\begin{align*}
	\|\lambda \widehat \bU(\lambda)\|_{\dot H^{0, 1}_q(\R^N_+)}
	&\le (\gamma p')^{-1/p'} \|\pd_t\bU\|_{L_p(\R_+, \dot H^{0, 1}_q(\R^N_+))}.
\end{align*}
Therefore, $\widehat \bU \in \CD(\CA_q)$ satisfies the resolvent problem:
\begin{equation}\label{resolvent U}
	\lambda \widehat \bU -\CA_q \widehat \bU = 0 \enskip \text{in $\R^N_+$}.
\end{equation}
Theorem \ref{thm:Rbdd H} implies that \eqref{resolvent U} has a unique solution for $\lambda \in \Sigma_\epsilon$; thus, we have $\widehat \bU(\lambda)=0$
for any $\lambda \in \C$ with $\gamma >0$.
Applying the Laplace inverse transform to $\widehat \bU(\lambda)=0$, we have
$\bV(t)=0$ for $t \in \R$. Therefore, we have $\bU(t)=0$ for $t>0$, which shows the uniqueness of \eqref{homo U}. 
Then $\nabla \fp=0$ also holds by the weak Dirichlet Neumann problem.
This completes the proof of Theorem \ref{thm:local mr}.
\end{proof}

\section{Weighted estimates}\label{sec:weight}
In this section, we prove the weighted estimates for the solutions of \eqref{linear}.
Let
\[
	\bF(t) = (\bff(t), \nabla \bG(t)), \quad
	\CF_q=\|(1+t)\bF(t) \|_{L_p(\R_+, L_q(\R^N_+))}.
\]
\begin{thm}\label{thm:weight}
Let $(\bu, \bQ, \fp)$ be a solution to \eqref{linear}
under the same assumption in Theorem \ref{thm:local mr}.
Then 
\begin{equation}\label{low infty} 
	\|(\bu, \bQ)\|_{L_\infty(\R_+, \dot H^{0, 1}_q(\R^N_+))} \le C(\|(\bu_0, \bQ_0)\|_{\dot H^{0, 1}_q(\R^N_+)}+\CF_q).
\end{equation}
In addition, let ${\widetilde q}$ be an index such that $1<{\widetilde q}<q$ and let $\kappa = N(1/{\widetilde q}-1/q) \le 1$.
If $1/p<\kappa/2$, the following estimates hold. 
\begin{equation}\label{low lp}
	\|(\bu, \bQ)\|_{L_p(\R_+, \dot H^{0, 1}_q(\R^N_+))}
	\le C \sum_{r\in \{q, {\widetilde q}\}}(\|(\bu_0, \bQ_0)\|_{\dot H^{0, 1}_r(\R^N_+)}+\CF_r), 
\end{equation}
\begin{equation}\label{weight:0}
\begin{aligned}
	&\|(1+t)(\pd_t, \nabla^2)(\bu, \bQ)\|_{L_p(\R_+, \dot H^{0, 1}_q(\R^N_+))}
	+\|(1+t)\nabla \bQ\|_{L_p(\R_+, L_q(\R^N_+))} +\|(1+t)\nabla \fp\|_{L_p(\R_+, L_q(\R^N_+))} \\
	&\enskip \le C (\|(\bu_0, \bQ_0)\|_{\CD_{q, p}(\R^N_+)} + \|(\bu_0, \bQ_0)\|_{\dot H^{0, 1}_{\widetilde q}(\R^N_+)}+\CF_q+\CF_{{\widetilde q}}).
\end{aligned}
\end{equation}
\end{thm}
To prove \eqref{weight:0}, we multiply \eqref{linear} with $t$,
\begin{equation}\label{weight}
\left\{
\begin{aligned}
	&\pd_t (t\bu) -\Delta (t\bu) + \nabla (t\fp) + \beta \DV (\Delta (t\bQ) -a (t\bQ))\\
	&= t \bff+\bu
	& \quad&\text{in $\R^N_+$}, \enskip t \in \R_+,\\
	&\dv (t\bu)=0& \quad&\text{in $\R^N_+$}, \enskip t \in \R_+,\\
	&\pd_t (t\bQ) - \beta \bD(t\bu) - \Delta (t\bQ) + a (t\bQ) = t\bG + \bQ& \quad&\text{in $\R^N_+$}, \enskip t \in \R_+,\\
	& t\bu=0, \enskip \pd_N (t\bQ) = 0 & \quad&\text{on $\R^N_0$}, \enskip t \in \R_+,\\
	&(t\bu, t\bQ)|_{t=0}=(0, 0)& \quad&\text{in $\R^N_+$}.
\end{aligned}
\right.
\end{equation}
Let $\bU=(\bu, \bQ)$. By \eqref{local}, we have
\begin{equation}\label{weight:1}
\begin{aligned}
	&\|t(\pd_t, \nabla^2)\bU\|_{L_p(\R_+, \dot H^{0, 1}_q(\R^N_+))}
	+\|t\nabla \bQ\|_{L_p(\R_+, L_q(\R^N_+))}
	+\|t\nabla \fp\|_{L_p(\R_+, L_q(\R^N_+))}\\
	&\le C(\|\pd_t (t\bU)\|_{L_p(\R_+, \dot H^{0, 1}_q(\R^N_+))}
	+\|\bU\|_{L_p(\R_+, \dot H^{0, 1}_q(\R^N_+))}
	+\|t \nabla^2 \bU \|_{L_p(\R_+, \dot H^{0, 1}_q(\R^N_+))}\\
	&\qquad 
	+\|t\nabla \bQ\|_{L_p(\R_+, L_q(\R^N_+))}
	+\|t\nabla \fp\|_{L_p(\R_+, L_q(\R^N_+))})\\
	&\le C(
	\CF_q
	 +\|\bU\|_{L_p(\R_+, \dot H^{0, 1}_q(\R^N_+))}),
\end{aligned}
\end{equation}
therefore, we need the estimate of the lower order term $\|\bU\|_{L_p(\R_+, \dot H^{0, 1}_q(\R^N_+))}$.
Assume that $(\bU_1, \fp_1)$ and $(\bU_2, \fp_2)$ satisfy \eqref{eq:U1} and \eqref{eq:U2}, respectively.

\subsection{Estimates of $\bU_1$}
Recall that
$\bU_1=(\bu_1, \bQ_1)$ satisfies
\begin{equation}\label{eq:U1'}
\left\{
\begin{aligned}
	&\pd_t\bu_1 -\Delta \bu_1 + \nabla \fp_1 + \beta \DV (\Delta \bQ_1 -a \bQ_1)=e_T[\bff],
	\enskip \dv \bu_1=0& \quad&\text{in $\R^N_+$}, \enskip t \in \R,\\
	&\pd_t \bQ_1 - \beta \bD(\bu_1) - \Delta \bQ_1 + a \bQ_1 =e_T[\bG]& \quad&\text{in $\R^N_+$}, \enskip t \in \R,\\
	&\bu_1 = 0, \enskip \pd_N \bQ_1 = 0& \quad&\text{on $\R^N_0$}, \enskip t \in \R,
\end{aligned}
\right.
\end{equation}
and also
\[
	\widetilde \bF(t)=(e_T[\bff], \nabla e_T[\bG]),
\]
where $e_T[f]$ is the extension of $f$ defined in \eqref{extension}. 
To obtain the estimate of $\bU_1$, we recall the semigroup $\{S(t)\}_{t \ge 0}$ associated with \eqref{eq:U1'}.
As we discussed in the proof of Theorem \ref{thm:local mr}, 
$\{S(t)\}_{t \ge 0}$ satisfies
	\begin{numcases}
{S(t) \widetilde \bF =}	 
	0 &\enskip \text{for $t<0$}, \nonumber \\
	\frac{1}{2\pi i} \int_{\Gamma_\omega} e^{\lambda t}
	(\CA(\lambda) \widetilde \bF, \CB(\lambda) \widetilde \bF)\,d\lambda
	&\enskip \text{for $t>0$}, \label{sec4 semi}\\
	\widetilde \bF &\enskip \text{for $t=0$}, \nonumber
	\end{numcases}
where $\Gamma_\omega = \Gamma_\omega^+ \cup \Gamma_\omega^- \cup C_\omega$
with
\[
\begin{aligned}
	\Gamma_\omega^\pm &= \{\lambda = re^{\pm i(\pi - \epsilon)} \mid \omega < r < \infty\},\\
	C_\omega &= \{\lambda = \omega e^{i\eta} \mid -(\pi - \epsilon) < \eta < (\pi - \epsilon)\}
\end{aligned}
\]
for $\omega >0$.
We also recall that
\begin{equation}\label{semi:0}
	\|S(t) \widetilde \bF\|_{\dot H^{0, 1}_q(\R^N_+)} \le M\|\widetilde \bF\|_{L_q(\R^N_+)},
\end{equation}
and besides,
$\bU_1$ can be represented by
\begin{equation}\label{U1}
\bU_1(t) =\int^t_{-\infty} S(t-s) \widetilde \bF(s)\,ds
\end{equation}
for any $t\in \R$. 

Now, we prove the decay estimates for $\{S(t)\}_{t \ge 0}$. The Gagliardo-Nirenberg inequality and \eqref{rem:resolvent} furnish that 
\begin{equation}\label{semi:1}
\begin{aligned}
	\|\CA(\lambda) \widetilde \bF\|_{L_q(\R^N_+)}
	&\le C \|\CA(\lambda) \widetilde \bF \|_{L_{{\widetilde q}}(\R^N_+)}^{1-\kappa}
	\|\nabla \CA(\lambda) \widetilde \bF \|_{L_{{\widetilde q}}(\R^N_+)}^{\kappa}\\
	&\le C|\lambda|^{-(1-\kappa/2)}\|\widetilde \bF \|_{L_{{\widetilde q}}(\R^N_+)}
\end{aligned}
\end{equation}
for $\lambda \in \Sigma_\epsilon$ provided that $1<{\widetilde q}<q$ and $\kappa=N(1/{\widetilde q}-1/q)$.
Similarly, we have
\begin{equation}\label{semi:2}
	\|\CB(\lambda) \widetilde \bF \|_{\dot H^1_q(\R^N_+)}
	\le C|\lambda|^{-(1-\kappa/2)}\|\widetilde \bF\|_{L_{{\widetilde q}}(\R^N_+)}.
\end{equation}
Thus, it holds that
\begin{equation}\label{semi:3}
	\|S(t) \widetilde \bF\|_{\dot H^{0, 1}_q(\R^N_+)}
	\le Ct^{-\kappa/2}\|\widetilde \bF\|_{L_{{\widetilde q}}(\R^N_+)}
\end{equation}
for $t>0$
if $1<{\widetilde q}<q$ and $\kappa=N(1/{\widetilde q}-1/q)$.
In fact, 
\eqref{sec4 semi} implies that 
$S(t) \widetilde \bF = S^+ (t) \widetilde \bF + S^- (t) \widetilde \bF + S^0(t) \widetilde \bF$,
where
\[
	S^\pm (t) \widetilde \bF = \frac{1}{2\pi i} \int_{\Gamma^\pm_\omega} e^{\lambda t}
	(\CA(\lambda) \widetilde \bF, \CB(\lambda) \widetilde \bF)\,d\lambda,\quad
	S^0 (t) \widetilde \bF = \frac{1}{2\pi i} \int_{C_\omega} e^{\lambda t}
	(\CA(\lambda) \widetilde \bF, \CB(\lambda) \widetilde \bF)\,d\lambda.
\]
It follows from \eqref{semi:1} and \eqref{semi:2} that
\begin{equation*}\label{spm}
\begin{aligned}
	\|S^\pm (t) \widetilde \bF\|_{\dot H^{0, 1}_q(\R^N_+)}
	&\le C \int^\infty_\omega e^{-(\cos \epsilon) tr}r^{-(1-\kappa/2)}\,dr \|\widetilde \bF\|_{L_{{\widetilde q}}(\R^N_+)}\\
	& \le C t^{-\kappa/2} \int^\infty_{t \omega} e^{-(\cos \epsilon) s}s^{-(1-\kappa/2)}\,ds \|\widetilde \bF\|_{L_{{\widetilde q}}(\R^N_+)},
\end{aligned}
\end{equation*}
where we have set $s=tr$.
Since $\omega$ is the arbitrary positive number, we choose $\omega = t^{-1}$,
then it holds that
\begin{equation}\label{spm}
	\|S^\pm (t) \widetilde \bF\|_{\dot H^{0, 1}_q(\R^N_+)}
	\le Ct^{-\kappa/2}\|\widetilde \bF\|_{L_{{\widetilde q}}(\R^N_+)}.
\end{equation}
Note that $|e^{\lambda t}| \le e^{|\omega e^{i\eta}t|} = e^{\omega t}$. 
Then we also have
\begin{align*}
	\|S^0 (t) \widetilde \bF\|_{\dot H^{0, 1}_q(\R^N_+)}
	&\le C e^{\omega t} \int^{\pi-\epsilon}_{-(\pi-\epsilon)} |\omega e^{i\eta}|^{-(1-\kappa/2)} \omega \,d\eta \|\widetilde \bF\|_{L_{{\widetilde q}}(\R^N_+)}\\
	& = C e^{\omega t} \int^{\pi-\epsilon}_{-(\pi-\epsilon)} \omega^{\kappa/2} \,d\eta \|\widetilde \bF\|_{L_{{\widetilde q}}(\R^N_+)}.
\end{align*}
Choosing $\omega = t^{-1}$,
it holds that
\begin{equation}\label{s0}
\begin{aligned}
	\|S^0 (t) \widetilde \bF\|_{\dot H^{0, 1}_q(\R^N_+)}
	&\le C e t^{-\kappa/2} \int^{\pi-\epsilon}_{-(\pi-\epsilon)} \,d\eta \|\widetilde \bF\|_{L_{{\widetilde q}}(\R^N_+)}\\
	&\le 2 \pi C e t^{-\kappa/2} \|\widetilde \bF\|_{L_{{\widetilde q}}(\R^N_+)}.
\end{aligned}
\end{equation}
Therefore, \eqref{semi:3} follows from \eqref{spm} and \eqref{s0}.

The following estimates for $\bU_1$ follow from \eqref{semi:0} and \eqref{semi:3}.
\begin{lem}\label{lem:U1}
Let $1< p, q< \infty$. Then
\begin{equation}\label{weight:2}
	\|\bU_1\|_{L_\infty(\R_+, \dot H^{0, 1}_q(\R^N_+))} \le C\CF_q.
\end{equation}
In addition, let ${\widetilde q}$ be an index such that $1<{\widetilde q}<q$ and let $\kappa=N(1/{\widetilde q}-1/q) \le 1$.
If $1/p<\kappa/2$,
\begin{equation}\label{weight:3}
	\|\bU_1\|_{L_p(\R_+, \dot H^{0, 1}_q(\R^N_+))} \le C (\CF_q+\CF_{{\widetilde q}}).
\end{equation}
\end{lem}

\begin{proof}
First, we prove \eqref{weight:2}.
Using \eqref{semi:0} for \eqref{U1} and applying H\"older's inequality, we have 
\begin{align*}
	\|\bU_1(t)\|_{\dot H^{0, 1}_q(\R^N_+)} 
	&\le M \int^t_{-\infty}\|\widetilde \bF(s)\|_{L_q(\R^N_+)}\,ds
	\le M \int^\infty_0 \|\bF(s)\|_{L_q(\R^N_+)}\,ds\\
	&\le C \|(1+s) \bF\|_{L_p(\R_+, L_q(\R^N_+))}=C\CF_q,
\end{align*}
which shows \eqref{weight:2}.

Second, we prove \eqref{weight:3}.
Set
\[
	\|\bU_1\|_{L_p((1, \infty), \dot H^{0, 1}_q(\R^N_+))}^p
	\le C (I+II),
\]
where
\begin{align*}
	I&=\int_1^\infty \left( \int^{t/2}_{-\infty} \|S(t-s) \widetilde \bF\|_{\dot H^{0, 1}_q(\R^N_+)}\,ds\right)^p\,dt,\\
	II&=\int_1^\infty \left( \int^t_{t/2} \|S(t-s) \widetilde \bF\|_{\dot H^{0, 1}_q(\R^N_+)}\,ds\right)^p\,dt.
\end{align*}
We consider the estimates of $I$ and $II$ by \eqref{semi:3}.
Noting that $(t-s) \ge t/2$ if $s<t/2$, 
we have
\begin{align*}
	I&\le C\int_1^\infty \left( \int^{t/2}_{-\infty} (t-s)^{-\frac{\kappa}{2}}\|\widetilde \bF\|_{L_{{\widetilde q}}(\R^N_+)}\,ds\right)^p\,dt\\
	&\le C\int_1^\infty t^{-\frac{p\kappa}{2}}\left( \int^{t/2}_{-\infty} \|\widetilde \bF\|_{L_{{\widetilde q}}(\R^N_+)}\,ds\right)^p\,dt\\
	&\le C\int_1^\infty t^{-\frac{p\kappa}{2}} \left( \int^{t/2}_{-\infty} (1+|s|)^{-p'}\,ds \right)^{p/p'} \,dt \|(1+|s|)\widetilde \bF\|_{L_p(\R, L_{{\widetilde q}}(\R^N_+))}^p\\
	&\le C\int_1^\infty t^{-\frac{p\kappa}{2}} \,dt \|(1+|s|)\widetilde \bF\|_{L_p(\R, L_{{\widetilde q}}(\R^N_+))}^p\\
	&\le C\|(1+s) \bF\|_{L_p(\R_+, L_{{\widetilde q}}(\R^N_+))}^p = C\CF_{{\widetilde q}}^p
	\end{align*}
provided that $1/p<\kappa/2$.
Furthermore,
\begin{align*}
	II&\le C\int_1^\infty \left( \int_{t/2}^t (t-s)^{-\frac{\kappa}{2}(\frac{1}{p'}+\frac{1}{p})}\|\widetilde \bF\|_{L_{{\widetilde q}}(\R^N_+)}\,ds\right)^p\,dt\\
	&\le C\int_1^\infty \left( \int_{t/2}^t (t-s)^{-\frac{\kappa}{2}}\,ds\right)^{p/p'}
	\left(\int^t_{t/2} (t-s)^{-\frac{\kappa}{2}}\|\widetilde \bF\|_{L_{{\widetilde q}}(\R^N_+)}^p\,ds\right)\,dt\\
	&\le C\int_1^\infty (t/2)^{\left(1-\frac{\kappa}{2}\right)\frac{p}{p'}}
	\left(\int^t_{t/2} (t-s)^{-\frac{\kappa}{2}}\|\widetilde \bF\|_{L_{{\widetilde q}}(\R^N_+)}^p\,ds\right)\,dt\\
	&\le C\int_{1/2}^\infty \int^{2s}_s t^{\left(1-\frac{\kappa}{2}\right)\frac{p}{p'}}
	(t-s)^{-\frac{\kappa}{2}}\,dt \|\widetilde \bF\|_{L_{{\widetilde q}}(\R^N_+)}^p\,ds\\
	&\le C\int_{1/2}^\infty s^{\left(1-\frac{\kappa}{2}\right)(p-1)+1-\frac{\kappa}{2}}
	\|\widetilde \bF\|_{L_{{\widetilde q}}(\R^N_+)}^p\,ds\\
	&\le C\|(1+s) \bF\|_{L_p(\R_+, L_{{\widetilde q}}(\R^N_+))}^p=C\CF_{{\widetilde q}}^p.
\end{align*}
Therefore, we have 
\begin{equation}\label{weight:4}
\|\bU_1\|_{L_p((1, T), \dot H^{0, 1}_q(\R^N_+))}
\le C \CF_{{\widetilde q}} 
\end{equation}
if $1<{\widetilde q}<q$, $\kappa=N(1/{\widetilde q}-1/q)$, and $1/p<\kappa/2$.

In addition, \eqref{semi:0} furnishes that
\begin{align*}
	\|\bU_1\|_{L_p((0, 1), \dot H^{0, 1}_q(\R^N_+))}^p
	&\le M \int^1_0 \left(\int^t_{-\infty} \|\widetilde \bF(s)\|_{L_q(\R^N_+)}\,ds\right)^p\,dt\\
	&\le M \int^1_0 \left(\int^\infty_{-\infty} (1+|s|)\|\widetilde \bF(s)\|_{L_q(\R^N_+)}\,ds\right)^p\,dt\\
	&\le M\|(1+s) \bF\|_{L_p(\R_+, L_q(\R^N_+))}^p=M\CF_q^p,
\end{align*}
together with \eqref{weight:4}, then we obtain \eqref{weight:3}.
\end{proof}

\subsection{Estimates of $\bU_2$}
Since $(\bu_1(0), \bQ_1(0))=(0, 0)$ by \eqref{U1(0)}, $\bU_2=(\bu_2, \bQ_2)$ satisfies
\begin{equation*}
\left\{
\begin{aligned}
	&\pd_t\bu_2 -\Delta \bu_2 + \nabla \fp_2 + \beta \DV (\Delta \bQ_2 -a \bQ_2)=0,
	\enskip \dv \bu_2=0& \quad&\text{in $\R^N_+$}, \enskip t \in \R_+,\\
	&\pd_t \bQ_2 - \beta \bD(\bu_2) - \Delta \bQ_2 + a \bQ_2 =0& \quad&\text{in $\R^N_+$}, \enskip t \in \R_+,\\
	&\bu_2 = 0, \enskip \pd_N \bQ_2 = 0 & \quad&\text{on $\R^N_0$}, \enskip t \in \R_+,\\
	&(\bu_2, \bQ_2)|_{t=0}=(\bu_0, \bQ_0)& \quad&\text{in $\R^N_+$}.
\end{aligned}
\right.
\end{equation*}
Let us consider the estimate for $\bU_2$.
\begin{lem}\label{lem:U2}
Let $1< p, q< \infty$. Then
\[
	\|\bU_2\|_{L_\infty(\R_+, \dot H^{0, 1}_q(\R^N_+))} \le C\|(\bu_0, \bQ_0)\|_{\dot H^{0, 1}_q(\R^N_+)}.
\]
In addition, let ${\widetilde q}$ be an index $1<{\widetilde q}<q$ and let $\kappa=N(1/{\widetilde q}-1/q) \le 1$.
If $1/p<\kappa/2$,
\[
	\|\bU_2\|_{L_p(\R_+, \dot H^{0, 1}_q(\R^N_+))} 
	\le C \sum_{r \in \{q, {\widetilde q}\}} \|(\bu_0, \bQ_0)\|_{\dot H^{0, 1}_r(\R^N_+)}.
\]
\end{lem}

\begin{proof}
Let $\bU_0=(\bu_0, \bQ_0)$.
Note that $\bU_2$ is represented by 
\[
	\bU_2(t)=T(t)\bU_0
\]
with
\[
	T(t)(\bff, \bG)=\frac{1}{2\pi i} \int_{\Gamma_\omega} e^{\lambda t}(\CA(\lambda)(\bff, \nabla \bG), 
	\CB(\lambda)(\bff, \nabla \bG))\,d\lambda
\]
for $t>0$, where $\Gamma_\omega$ is defined in \eqref{int path} for $\omega>0$.
By the same manner as in the proof of \eqref{semi:0} and \eqref{semi:3}, we have
\begin{align}
	&\|\bU_2(t)\|_{\dot H^{0, 1}_q(\R^N_+)} \le C\|\bU_0\|_{\dot H^{0, 1}_q(\R^N_+)}, \label{semi:4}\\
	&\|\bU_2(t)\|_{\dot H^{0, 1}_q(\R^N_+)} \le Ct^{-\kappa/2} \|\bU_0\|_{\dot H^{0, 1}_{{\widetilde q}}(\R^N_+)} \label{semi:5}
\end{align}
if $1<{\widetilde q}<q$ and $\kappa=N(1/{\widetilde q}-1/q)$.
Here, \eqref{semi:4} implies that
\begin{align}
	&\|\bU_2(t)\|_{L_\infty(\R_+, \dot H^{0, 1}_q(\R^N_+))} 
	\le C \|\bU_0\|_{\dot H^{0, 1}_q(\R^N_+)}, \nonumber \\
	&\|\bU_2(t)\|_{L_p((0, 1), \dot H^{0, 1}_q(\R^N_+))} 
	\le C \|\bU_0\|_{\dot H^{0, 1}_q(\R^N_+)}.\label{U2:1}
\end{align}
Furthermore, 
\eqref{semi:5} furnishes that
\begin{equation}\label{U2:2}
	\|\bU_2(t)\|_{L_p((1, \infty), \dot H^{0, 1}_q(\R^N_+))} 
	\le C \|\bU_0\|_{\dot H^{0, 1}_{{\widetilde q}}(\R^N_+)}
\end{equation}
if $1/p<\kappa/2$.
Thus, by \eqref{U2:1} and \eqref{U2:2} we have
\begin{equation*}\label{U2:3}
	\|\bU_2(t)\|_{L_p(\R_+, \dot H^{0, 1}_q(\R^N_+))} 
	\le C \sum_{r \in \{q, {\widetilde q}\}}\|\bU_0\|_{\dot H^{0, 1}_r(\R^N_+)},
	\end{equation*}
which completes the proof of
Lemma \ref{lem:U2}.
\end{proof}

\subsection{Proof of Theorem \ref{thm:weight}}
Recall that $\bU=\bU_1+\bU_2$.
Lemma \ref{lem:U1} and Lemma \ref{lem:U2} furnish that
\[
	\|\bU\|_{L_\infty(\R_+, \dot H^{0, 1}_q(\R^N_+))} \le C(\|\bU_0\|_{\dot H^{0, 1}_q(\R^N_+)}+\CF_q)
\]
and 
\begin{equation}\label{weight:5}
	\|\bU\|_{L_p(\R_+, \dot H^{0, 1}_q(\R^N_+))} 
	\le C \sum_{r \in \{q, {\widetilde q}\}} (\|\bU_0\|_{\dot H^{0, 1}_r(\R^N_+)}+\CF_r)
\end{equation}
if $1<{\widetilde q}<q$, $\kappa=N(1/{\widetilde q}-1/q)$, and $1/p<\kappa/2$, which prove \eqref{low infty} and \eqref{low lp}.
Combining \eqref{local}, \eqref{weight:1}, and \eqref{weight:5},
we have \eqref{weight:0}.

\section{Global well-posedness}\label{sec:global}
In this section, let us prove Theorem \ref{thm:global}.
Let $\bU=(\bu, \bQ)$ and $\bU_0=(\bu_0, \bQ_0)$.
Hereafter, we may assume that $0<\sigma<1$.
Theorem \ref{thm:global} is proved by the Banach fixed point argument. 
The uniqueness of the solutions follows from the uniqueness of the fixed points; therefore, we focus on the existence of solutions.

Let $N \ge 2$ and $0<\theta<1/2$. Note that
the assumption \eqref{condi} implies that 
\[
	\frac{1}{q_0}=\frac{1}{q_1}+\frac{1}{q_2}, \quad N\left(\frac{1}{q_1}-\frac{1}{q_2}\right)=1, \quad \frac{1-\theta}{q_1}+\frac{\theta}{q_2}=\frac{1}{N}, \quad 1<q_0<q_1<N<q_2<\infty,
\]
where $q_0=N/(1+2\theta) \ge 2/(1+2\theta) >1$ follows from $N \ge 2$ and $0<\theta<1/2$.
Recall that
\begin{align*}
	E(\bU)&=\sum^2_{i=1}(\|(1+t)(\pd_t, \nabla^2)\bU\|_{L_p(\R_+, \dot H^{0, 1}_{q_i}(\R^N_+))}+\|(1+t)\nabla \bQ\|_{L_p(\R_+, L_{q_i}(\R^N_+)}\\
	&\quad +\|\bU\|_{L_p(\R_+, \dot H^{0, 1}_{q_i}(\R^N_+))}
	+\|\bU\|_{L_\infty(\R_+, \dot H^{0, 1}_{q_i}(\R^N_+))}).
\end{align*}
Define the underlying space as
\begin{align*}
	\CI_{\sigma} = \{\bU \mid & \pd_t \bu \in \bigcap^2_{i=1} L_p(\R_+, L_{q_i}(\R^N_+)), \enskip \bu \in \bigcap^2_{i=1} L_p(\R_+, \dot H^2_{q_i}(\R^N_+)),\\
	&\pd_t \bQ \in \bigcap^2_{i=1} L_p (\R_+, \dot H^1_{q_i}(\R^N_+; \bS_0)), \enskip
	\bQ \in \bigcap^2_{i=1} L_p (\R_+, \dot H^1_{q_i}(\R^N_+; \bS_0) \cap \dot H^3_{q_i}(\R^N_+; \bS_0)),\\
	& \bU|_{t=0} = \bU_0, \enskip E(\bU) \le \sigma \}.
\end{align*}
Given $\bU \in \CI_{\sigma}$, we assume that $\bV = (\bv, \bP)$ satisfies 
\begin{equation}\label{eq:fixed point}
\left\{
\begin{aligned}
	&\pd_t\bv -\Delta \bv + \nabla \fp + \beta \DV (\Delta \bP -a \bP)=\bff(\bU),
	\enskip \dv \bv=0& \quad&\text{in $\R^N_+$}, \enskip t \in \R_+,\\
	&\pd_t \bP - \beta \bD(\bv) - \Delta \bP + a \bP =\bG(\bU)& \quad&\text{in $\R^N_+$}, \enskip t \in \R_+,\\
	&\bv = 0, \enskip \pd_N \bP = 0& \quad&\text{on $\R^N_0$}, \enskip t \in \R_+,\\
	&\bV|_{t=0}=\bU_0& \quad&\text{in $\R^N_+$}.
\end{aligned}
\right.
\end{equation}
Let us prove $\bV \in \CI_{\sigma}$.
To achieve that, we show
\begin{equation}\label{est:fG}
\begin{aligned}
	\sum_{r\in\{q_0, q_1, q_2\}} (\|(1+t) \bff(\bU)\|_{L_p(\R_+, L_r(\R^N_+))}+\|(1+t) \bG(\bU)\|_{L_p(\R_+, \dot H^1_r(\R^N_+))})
	\le C \sigma^2
\end{aligned}
\end{equation} 
by using the following lemma proved by \cite{MPS}.
\begin{lem}\label{sup}
Let $\ell =0, 1$ and $N \ge 2$. Let $1 < q_1<N<q_2 < \infty$, and let $0 < \theta <1$.
Assume that
\[
	\frac{1-\theta}{q_1}+\frac{\theta}{q_2}=\frac{1}{N}.
\]
Then
\[
	\|\nabla^\ell \bv\|_{L_\infty(\R^N_+)} \le C \|\nabla^{\ell+1} \bv\|_{L_{q_1}(\R^N_+)}^{1-\theta} \|\nabla^{\ell+1} \bv\|_{L_{q_2}(\R^N_+)}^\theta 
\]
for $\bv \in \dot{H}^{\ell+1}_{q_1}(\R^N_+)^N \cap \dot{H}^{\ell+1}_{q_2}(\R^N_+)^N$.
\end{lem}
Let us consider the estimate of $\bff(\bU)$.
For $i=1, 2$, it holds by Lemma \ref{sup} that
\begin{align*}
	\|\bu \cdot \nabla \bu\|_{L_{q_i}(\R^N_+)} 
	&\le C\|\bu\|_{L_{q_i}(\R^N_+)}\|\nabla \bu\|_{L_\infty(\R^N_+)}\\
	&\le C\|\bu\|_{L_{q_i}(\R^N_+)}\|\nabla^2 \bu\|_{L_{q_1}(\R^N_+)}^{1-\theta} \|\nabla^2 \bu\|_{L_{q_2}(\R^N_+)}^{\theta}\\
	&\le C\|\bu\|_{L_{q_i}(\R^N_+)}(\|\nabla^2 \bu\|_{L_{q_1}(\R^N_+)} + \|\nabla^2 \bu\|_{L_{q_2}(\R^N_+)})
\end{align*}
if $(1-\theta)/q_1 + \theta/q_2 = 1/N$.
Then
\begin{align*}
	&\|(1+t)\bu \cdot \nabla \bu\|_{L_p(\R_+, L_{q_i}(\R^N_+))}\\
	&\le C\|\bu\|_{L_\infty(\R_+, L_{q_i}(\R^N_+))} (\|(1+t)\nabla^2 \bu\|_{L_p(\R_+, L_{q_1}(\R^N_+))} + \|(1+t)\nabla^2 \bu\|_{L_p(\R_+, L_{q_2}(\R^N_+))})\\
	&\le CE(\bU)^2.
\end{align*}
Note that $1/q_0=1/q_1+1/q_2$ and $N(1/q_1-1/q_2)=1$.
H\"older's inequality and Sobolev's embedding theorem imply that
\begin{align*}
	\|\bu \cdot \nabla \bu\|_{L_{q_0}(\R^N_+)} 
	&\le C\|\bu\|_{L_{q_1}(\R^N_+)}\|\nabla \bu\|_{L_{q_2}(\R^N_+)}\\
	&\le C\|\bu\|_{L_{q_1}(\R^N_+)}\|\nabla^2 \bu\|_{L_{q_1}(\R^N_+)},
\end{align*}
then
\begin{align*}
	&\|(1+t)\bu \cdot \nabla \bu\|_{L_p(\R_+, L_{q_0}(\R^N_+))}\\
	&\le C\|\bu\|_{L_\infty(\R_+, L_{q_1}(\R^N_+))} \|(1+t)\nabla^2 \bu\|_{L_p(\R_+, L_{q_1}(\R^N_+))}\\
	&\le CE(\bU)^2.
\end{align*}
Note that other terms of $\bff(\bU)$ are written by $\bQ^k P(\bQ)$ with $k=0, 1, 2, 3$, where $P(\bQ) = (\nabla \bQ \nabla^2 \bQ, \nabla^3 \bQ \bQ, \bQ \nabla \bQ)$.
{It holds by Lemma \ref{sup} and Young's inequality that
\[
	\|\bQ\|_{L_\infty(\R^N_+)} \le C(\|\nabla \bQ\|_{L_{q_1}(\R^N_+)} + \|\nabla \bQ\|_{L_{q_2}(\R^N_+)}),
\]
then we have
\begin{align*}
	&\|(1+t) \bQ^k P(\bQ)\|_{L_p(\R_+, L_q(\R^N_+))} \\
	&\enskip \le \|\bQ\|_{L_\infty(\R_+, L_\infty(\R^N_+))}^k \|(1+t) P(\bQ)\|_{L_p(\R_+, L_q(\R^N_+))}\\
	&\enskip \le C(\|\nabla \bQ\|_{L_\infty(\R_+, L_{q_1}(\R^N_+))} ^k+ \|\nabla \bQ\|_{L_\infty(\R_+, L_{q_2}(\R^N_+))}^k) \|(1+t) P(\bQ)\|_{L_p(\R_+, L_q(\R^N_+))}.
\end{align*}
Therefore,
it is sufficient to consider the estimate of $\|(1+t) P(\bQ)\|_{L_p(\R_+, L_q(\R^N_+))}$.
It follows from the same manner as $\bu \cdot \nabla \bu$ that
\begin{align*}
	&\|(1+t)\nabla \bQ \nabla^2 \bQ\|_{L_p(\R_+, L_{q_i}(\R^N_+))}\\
	&\le C\|\nabla \bQ\|_{L_\infty(\R_+, L_{q_i}(\R^N_+))} (\|(1+t)\nabla^3 \bQ\|_{L_p(\R_+, L_{q_1}(\R^N_+))} + \|(1+t)\nabla^3 \bQ\|_{L_p(\R_+, L_{q_2}(\R^N_+))})\\
	&\le CE(\bU)^2,\\
	&\|(1+t)\nabla \bQ \nabla^2 \bQ\|_{L_p(\R_+, L_{q_0}(\R^N_+))}\\
	&\le C\|\nabla \bQ\|_{L_\infty(\R_+, L_{q_1}(\R^N_+))} \|(1+t)\nabla^3 \bQ\|_{L_p(\R_+, L_{q_1}(\R^N_+))}\\
	&\le CE(\bU)^2
\end{align*}
if $1/q_0=1/q_1+1/q_2$, $N(1/q_1-1/q_2)=1$, and $(1-\theta)/q_1+\theta/q_2=1/N$.
Furthermore, for $i=1, 2$, it holds by Lemma \ref{sup} that
\[
	\|\nabla^3 \bQ \bQ\|_{L_{q_i}(\R^N_+)} \le C \|\nabla^3 \bQ\|_{L_{q_i}(\R^N_+)} \|\bQ\|_{L_\infty(\R^N_+)}
	 \le C \|\nabla^3 \bQ\|_{L_{q_i}(\R^N_+)} (\|\nabla \bQ\|_{L_{q_1}(\R^N_+)}+\|\nabla \bQ\|_{L_{q_2}(\R^N_+)})
\] provided that $(1-\theta)/q_1+\theta/q_2=1/N$, 
then
\begin{align*}
	&\|(1+t)\nabla^3 \bQ \bQ\|_{L_p(\R_+, L_{q_i}(\R^N_+))}\\
	&\le C\|(1+t)\nabla^3 \bQ\|_{L_p(\R_+, L_{q_i}(\R^N_+))} (\|\nabla \bQ\|_{L_\infty(\R_+, L_{q_1}(\R^N_+))}+\|\nabla \bQ\|_{L_\infty(\R_+, L_{q_2}(\R^N_+))})\\
	&\le CE(\bU)^2.
\end{align*}
By the same way as the estimate of $\bu \cdot \nabla \bu$, we have
\begin{align*}
	&\|(1+t)\nabla^3 \bQ \bQ\|_{L_p(\R_+, L_{q_0}(\R^N_+))}\\
	&\le C \|(1+t)\nabla^3 \bQ\|_{L_p(\R_+, L_{q_1}(\R^N_+))} \|\nabla \bQ\|_{L_\infty(\R_+, L_{q_1}(\R^N_+))}\\
	&\le CE(\bU)^2.
\end{align*}
Repeating the same manner as before, Lemma \ref{sup} gives us
\[
	\|\bQ \nabla \bQ\|_{L_{q_i}(\R^N_+)} \le C \|\nabla \bQ\|_{L_{q_i}(\R^N_+)} \|\bQ\|_{L_\infty(\R^N_+)}
	 \le C \|\nabla \bQ\|_{L_{q_i}(\R^N_+)} (\|\nabla \bQ\|_{L_{q_1}(\R^N_+)}+\|\nabla \bQ\|_{L_{q_2}(\R^N_+)})
\]
for $i=1, 2$
provided that $(1-\theta)/q_1+\theta/q_2=1/N$, then we have
\[
\begin{aligned}
	&\|(1+t) \bQ \nabla \bQ\|_{L_p(\R_+, L_{q_i}(\R^N_+))}\\
	&\le C \|(1+t)\nabla \bQ\|_{L_p(\R_+, L_{q_i}(\R^N_+))} (\|\nabla \bQ\|_{L_\infty(\R_+, L_{q_1}(\R^N_+))}+\|\nabla \bQ\|_{L_\infty(\R_+, L_{q_2}(\R^N_+))})\\
	&\le C E(\bU)^2
\end{aligned}
\]
for $i=1, 2$.
Furthermore, H\"older's inequality, Sobolev's embedding theorem, and Lemma \ref{sup} imply that
\begin{align*}
	&\|(1+t) \bQ \nabla \bQ\|_{L_p(\R_+, L_{q_0}(\R^N_+))}\\
	&\le C \|(1+t)\bQ\|_{L_p(\R_+, L_{q_2}(\R^N_+))} \|\nabla \bQ\|_{L_\infty(\R_+, L_{q_1}(\R^N_+))}\\
	&\le C \|(1+t)\nabla \bQ\|_{L_p(\R_+, L_{q_1}(\R^N_+))} \|\nabla \bQ\|_{L_\infty(\R_+, L_{q_1}(\R^N_+))}\\
	&\le C E(\bU)^2
\end{align*}
if $1/q_0=1/q_1+1/q_2$, $N(1/q_1-1/q_2)=1$, and $(1-\theta)/q_1+\theta/q_2=1/N$.
Since we can estimate $\bG(\bU)$
in the same manner,
we have 
\begin{equation}\label{est:nonlinear}
\begin{aligned}
	&\sum_{r\in\{q_0, q_1, q_2\}} \|(1+t) \bff(\bU)\|_{L_p(\R_+, L_r(\R^N_+))}
	\le C \sum_{k=2}^5 E(\bU)^k,\\
	&\sum_{r\in\{q_0, q_1, q_2\}} \|(1+t) \bG(\bU)\|_{L_p(\R_+, \dot H^1_r(\R^N_+))} 
	\le C \sum_{k=2}^3 E(\bU)^k.
\end{aligned}
\end{equation}
It holds by $\bU \in \CI_{\sigma}$ and $0 < \sigma <1$ that \eqref{est:fG}.

Now, we can apply Theorem \ref{thm:local mr} to \eqref{eq:fixed point}, then 
we observe that there exists a solution $(\bv, \bP, \fp)$ of \eqref{eq:fixed point} with
\begin{align*}
	\pd_t \bv &\in \bigcap^2_{i=1} L_p (\R_+, L_{q_i}(\R^N_+)^N), & 
	\bv&\in L_p (\R_+, \dot H^2_{q_i}(\R^N_+)^N),\\
	\pd_t \bP &\in \bigcap^2_{i=1} L_p (\R_+, \dot H^1_{q_i}(\R^N_+; \bS_0)), & 
	\bP &\in L_p (\R_+, \dot H^1_{q_i}(\R^N_+; \bS_0) \cap \dot H^3_{q_i}(\R^N_+; \bS_0)),\\
	\nabla \fp &\in \bigcap^2_{i=1} L_p (\R_+, L_{q_i}(\R^N_+)^N). & &
\end{align*}
Theorem \ref{thm:weight} works for $q_0$, $q_1$, $q_2$, and $p$ satisfying \eqref{condi}.
In fact, 
$q_0$, $q_1$, and $p$ satisfy $1<q_0<q_1$, $N(1/q_0 - 1/q_1) = \theta$, and $1/p < \theta/2$ for $0<\theta<1/2$; therefore, 
Theorem \ref{thm:weight} holds for $(\kappa, {\widetilde q}, q)=(\theta, q_0, q_1)$.
Furthermore, since $q_1$, $q_2$, and $p$ satisfy $1<q_1<q_2$, $N(1/q_1 - 1/q_2) = 1$, and $1/p < 1/2$, Theorem \ref{thm:weight} works for $(\kappa, {\widetilde q}, q)=(1, q_1, q_2)$.
Therefore, Theorem \ref{thm:weight}, together with \eqref{est:fG} and \eqref{smallness condi1}, enables us to obtain 
\begin{equation*}
\begin{aligned}
	E(\bV) &\le C\Big(\sum^2_{i=1} \|\bU_0\|_{\CD_{q_i, p}(\R^N_+)}
	+\|\bU_0\|_{\dot H^{0, 1}_{q_0}(\R^N_+)}\\
	&\enskip+\sum_{r\in\{q_0, q_1, q_2\}} (\|(1+t) \bff(\bU)\|_{L_p(\R_+, L_r(\R^N_+))}+\|(1+t) \bG(\bU)\|_{L_p(\R_+, \dot H^1_r(\R^N_+))})\Big)\\
	& \le C \sigma^2
\end{aligned}
\end{equation*}
provided that \eqref{condi}.
Choosing $\sigma>0$ so small that $C\sigma<1$, we have
\[
	E(\bV) \le \sigma,
\]
which implies that $\bV \in \CI_{\sigma}$.
Define a solution map $\Phi$ as $\Phi(\bU) = \bV$, then $\Phi$ maps from $\CI_{\sigma}$ into itself.

Next, we prove the map $\Phi$ is a contraction map; namely, 
it holds that there exists $\delta \in (0, 1)$ such that
\begin{equation}\label{contract}
	E(\Phi(\bU_1) - \Phi(\bU_2)) \le \delta E(\bU_1 - \bU_2)
\end{equation}
for any $\bU_1, \bU_2 \in \CI_\sigma$.
Let $\Phi(\bU_i) = \bV_i = (\bv_i, \bP_i)$ for $i=1, 2$.
Set $\bV = (\bv, \bP) = (\bv_1, \bP_1) - (\bv_2, \bP_2)$ and $\fp = \fp_1 - \fp_2$.
Then $(\bV, \fp)$ is a solution of the following problem.
\[
\left\{
\begin{aligned}
	&\pd_t\bv -\Delta \bv + \nabla \fp + \beta \DV (\Delta \bP -a \bP)=\bff(\bU_1) - \bff(\bU_2),
	\enskip \dv \bv=0& \quad&\text{in $\R^N_+$}, \enskip t \in \R_+,\\
	&\pd_t \bP - \beta \bD(\bv) - \Delta \bP + a \bP =\bG(\bU_1) - \bG(\bU_2) & \quad&\text{in $\R^N_+$}, \enskip t \in \R_+,\\
	&\bv = 0, \enskip \pd_N \bP = 0& \quad&\text{on $\R^N_0$}, \enskip t \in \R_+,\\
	&\bV|_{t=0}=0& \quad&\text{in $\R^N_+$}.
\end{aligned}
\right.
\]
In addition, it follows from Theorem \ref{thm:weight} that
\begin{equation}\label{Ev}
\begin{aligned}
	E(\bV) 
	\le \sum_{r\in\{q_0, q_1, q_2\}} &(\|(1+t) (\bff(\bU_1)-\bff(\bU_2))\|_{L_p(\R_+, L_r(\R^N_+))}\\
	&\enskip
	+\|(1+t) (\bG(\bU_1)-\bG(\bU_2))\|_{L_p(\R_+, \dot H^1_r(\R^N_+))}).
\end{aligned}
\end{equation}
By the same calculation that yields \eqref{est:nonlinear}, we have 
\begin{equation}\label{est:contract}
\begin{aligned}
	&\sum_{r\in\{q_0, q_1, q_2\}} \|(1+t) (\bff(\bU_1) - \bff(\bU_1))\|_{L_p(\R_+, L_r(\R^N_+))}
	\le C \sum_{k=1}^4 (E(\bU_1) + E(\bU_2))^k E(\bU_1 - \bU_2),\\
	&\sum_{r\in\{q_0, q_1, q_2\}} \|(1+t) (\bG(\bU_1) - \bG(\bU_2))\|_{L_p(\R_+, \dot H^1_r(\R^N_+))}  
	\le C \sum_{k=1}^2 (E(\bU_1) + E(\bU_2))^k E(\bU_1 - \bU_2)
\end{aligned}
\end{equation}
under the condition \eqref{condi}.
In fact, for instance, we consider
\[
	(\bu_1 \cdot \nabla)\bu_1 - (\bu_2 \cdot \nabla)\bu_2
	= ((\bu_1 - \bu_2) \cdot \nabla)\bu_1 - (\bu_2 \cdot \nabla)(\bu_1 - \bu_2).
\]
One can again use Lemma \ref{sup} and obtain that
\begin{align*}
	&\|(1+t)((\bu_1 - \bu_2) \cdot \nabla)\bu_1\|_{L_p(\R_+, L_{q_i}(\R^N_+))}\\ 
	&\le C \|\bu_1-\bu_2\|_{L_\infty(\R_+, L_{q_i}(\R^N_+))} \sum_{r \in \{q_1, q_2\}} \|(1+t)\nabla^2 \bu_1\|_{L_p(\R_+, L_r(\R^N_+))}
	\le CE(\bU_1 - \bU_2) E(\bU_1),\\
	&\|(1+t)(\bu_2 \cdot \nabla)(\bu_1 - \bu_2)\|_{L_p(\R_+, L_{q_i}(\R^N_+))}\\ 
	&\le C \|\bu_2\|_{L_\infty(\R_+, L_{q_i}(\R^N_+))} \sum_{r \in \{q_1, q_2\}} \|(1+t)\nabla^2 (\bu_1 - \bu_2)\|_{L_p(\R_+, L_r(\R^N_+))}
	\le CE(\bU_2) E(\bU_1 - \bU_2)
\end{align*}
for $i=1, 2$ if $(1-\theta)/q_1 + \theta/q_2 = 1/N$.
Therefore, we have
\[
	\|(\bu_1 \cdot \nabla)\bu_1 - (\bu_2 \cdot \nabla)\bu_2\|_{L_p(\R_+, L_{q_i}(\R^N_+))}
	\le C(E(\bU_1)+E(\bU_2)) E(\bU_1 - \bU_2).
\]
Repeating similar computations, we arrive at \eqref{est:contract}. 
It holds from  \eqref{Ev} and \eqref{est:contract} that
\[
	E(\bV) \le C \sum_{k=1}^4 (E(\bU_1) + E(\bU_2))^k E(\bU_1 - \bU_2).
\]
Choosing $\sigma>0$ so small that $C \sum_{k=1}^4 (E(\bU_1) + E(\bU_2))^k < \delta$, we have \eqref{contract}.

Therefore, the Banach fixed point argument indicates that there exists a unique solution $\bV \in \CI_{\sigma}$ such that $\Phi(\bV) = \bV$, namely,
$(\bV, \fp)$ is a unique solution of \eqref{nonlinear}.

The weighted estimate of $\nabla \fp$ follows from the first equation of \eqref{nonlinear}.
Since $\bV=(\bv, \bP) \in \CI_\sigma$ is a solution of \eqref{nonlinear} and $\|(1+t)\bff(\bV)\|_{L_p(\R_+, L_{q_i}(\R^N_+))} \le \sigma$, we have
\begin{align*}
	&\|(1+t)\nabla \fp\|_{L_p(\R_+, L_{q_i}(\R^N_+))} \\
	&\le C\|(1+t)(\pd_t \bv, \nabla^2 \bv)\|_{L_p(\R_+, L_{q_i}(\R^N_+))} + \|(1+t)(\nabla^3 \bP, \nabla \bP)\|_{L_p(\R_+, L_{q_i}(\R^N_+))}
	+ \|(1+t) \bff(\bV)\|_{L_p(\R_+, L_{q_i}(\R^N_+))}\\
	&\le C\sigma
\end{align*}
for $i=1, 2$, which completes the proof of Theorem \ref{thm:global}.


\begin{thebibliography}{9}
{\small 
\bibitem{Ab1} H.~Abels, G.~Dolzmann, Y.~Liu,
{\it Well-Posedness of a Fully Coupled Navier-Stokes/Q-tensor System with Inhomogeneous Boundary Data},
SIAM J. Math. Anal., {\bf 46 (4)} (2014) 3050--3077.  
\bibitem{Ab2} H.~Abels, G.~Dolzmann, Y.~Liu,
{\it Strong solutions for the Beris-Edwards model for nematic liquid crystals with homogeneous Dirichlet boundary conditions},
Adv. Differential Equations, {\bf 21 (1)-(2)} (2016) 109--152.
\bibitem{BM} D.~Barbera and M.~Murata,
{\it The $L^p$-$L^q$ maximal regularity for the Beris-Edward model in the half-space},
Ann. Sc. Norm. Super. Pisa Cl. Sci. {\bf 56} (2024).
\bibitem{BM2025} D.~Barbera, M.~Murata,
{\it The $\CR$-boundedness of solution operators for the Q-tensor model of nematic liquid crystals in $\R^N$ and $\R^N_+$},
preprint, Available at https://arxiv.org/abs/2506.05152v1.
\bibitem{BL}J.~Bergh and J.~L\"ofstr\"om. 
{Interpolation spaces. An introduction.} Grundlehren der Mathematischen
 Wissenschaften, {\bf 223}, Springer-Verlag, Berlin-New York (1976).
\bibitem{BE} A.~N. Beris and B.~J.~Edwards,
{Thermodynamics of Flowing Systems with Internal Microstructure}, 
Oxford Engrg. Sci. Ser. {\bf 36}, Oxford University Press, Oxford, New York
(1994).
\bibitem{DHMT} R.~Danchin, M.~Hieber, P.~B.~Mucha, and P.~Tolksdorf, 
{\it Free Boundary Problems by Da Prato – Grisvard Theory},
Mem. Amer. Math. Soc. {\bf 311} (2025) no. 1578.
\bibitem{D} F.~De Anna, 
{\it A global 2D well-posedness result on the order tensor liquid crystal theory},
J. Differential Equations {\bf 262 (7)} (2017) 3932--3979.
\bibitem{DG}P.~G.~De Gennes and J.~Prost,
{The Physics of Liquid Crystals}, Oxford University Press, Oxford, New York (1993).
\bibitem{H} M.~Haase. {\it The functional calculus for sectorial operators}, 
Operator Theory: Advances and
Applications {\bf 169}, Birkh\"auser Verlag, Basel (2006).
\bibitem{HHW} M.~Hieber, A.~Hussein, and M.~Wrona, 
{\it Strong well-posedness of the Q-tensor model for liquid crystals: the case of arbitrary ratio of tumbling and aligning effects $\xi$}, 
Arch. Ration. Mech. Anal. {\bf 248} (2024) no. 3, Paper No. 40.
\bibitem{HD}J.~Huang and S.~Ding,
{\it Global well-posedness for the dynamical $Q$-tensor model of liquid crystals},
Science China Mathematics {\bf 58} (2015) 1349--1366.
\bibitem{K} T.~Kato,
{Perturbation Theory for Linear Operators},
Springer, Berlin (1995).
\bibitem{LW} Y.~Liu and W.~Wang, 
{\it On the initial boundary value problem of a Navier-Stokes/$Q$-tensor model for liquid crystals},
Discrete Contin. Dyn. Syst. Ser. {\bf B 23 (9)} (2018) 3879--3899.
\bibitem{M} A.~Majumdar, 
{\it Equilibrium order parameters of liquid crystals in the Landau-de Gennes theory},
European J. Appl. Math. {\bf 21} (2010) 181--203.
\bibitem{MPS} P.B.~Mucha, T.~Piasecki, and Y,~Shibata,
preprint.
\bibitem{MS} M.~Murata and Y.~Shibata, 
{\it Global well posedness for a Q-tensor model of nematic liquid crystals},
J. Math. Fluid Mech. {\bf 24 (1)} (2022) Paper No. 34.
\bibitem{PZ1} M.~Paicu, and A.~Zarnescu,
{\it Global existence and regularity for the full coupled Navier-Stokes and $Q$-tensor system}, 
SIAM J. Math. Anal. {\bf 43 (5)} (2011) 2009--2049.
\bibitem{PZ2} M.~Paicu, and A.~Zarnescu,
{\it Energy dissipation and regularity for a coupled Navier-Stokes and $Q$-tensor system}, 
Arch. Ration. Mech. Anal. {\bf 203 (1)} (2012) 45--67.
\bibitem{SS} M.~Schonbek and Y.~Shibata, 
{\it Global well-posedness and decay for a 
$\mathbb Q$
tensor model of incompressible nematic liquid crystals in  
$\R^N$},
J. Differential Equations {\bf 266 (6)} (2019) 3034--3065. 
\bibitem{SS2} Y.~Shibata and S.~Shimizu, 
{\it On the $L_p$-$L_q$ maximal regularity of the Neumann 
problem for the Stokes equations in a bounded domain},
J.~Reine Angew. Math. {\bf 615} (2008) 157--209.
\bibitem{SW} Y.~Shibata and K.~Watanabe, 
{\it Maximal $L_1$-regularity of the Navier-Stokes equations with free boundary conditions via a generalized semigroup theory},
J. Differential Equations {\bf 426} (2025) 495--605.
\bibitem{T} H.~Tanabe,
{\it Functional Analytic Methods for Partial Differential Equations},
Monographs and Textbooks in Pure and Applied Mathematics {\bf 204}, Marcel Dekker, Inc., New York, Basel (1997).
\bibitem{W} L.~Weis, 
{\it Operator-valued Fourier multiplier 
theorems and maximal $L_p$-regularity}. 
Math. Ann. {\bf 319} (2001) 735--758.
\bibitem{X} Y.~Xiao,
{\it Global strong solution to the three-dimensional liquid crystal flows of Q-tensor model},
J. Differential Equations {\bf 262 (3)} (2017) 1291--1316.
}
\end{thebibliography}
\end{document}